\documentclass[11pt,reqno]{amsart}
\usepackage{amscd,amsfonts,amsmath,amssymb,amstext,amsthm}
\usepackage[all]{xy} 
\usepackage{xcolor}
\usepackage{enumerate}
\usepackage{blkarray}
\usepackage[linktocpage=true]{hyperref}
\usepackage{latexsym}

\usepackage{graphicx} 
\usepackage{environ}         
\usepackage{etoolbox}        

\newlength{\myl}
\let\origequation=\equation
\let\origendequation=\endequation

\RenewEnviron{equation}{
  \settowidth{\myl}{$\BODY$}                       
  \origequation
  \ifdimcomp{\the\linewidth}{>}{\the\myl}
  {\ensuremath{\BODY}}                             
  {\resizebox{\linewidth}{!}{\ensuremath{\BODY}}}  
  \origendequation
}

\makeatletter
\theoremstyle{plain}
\newtheorem{theorem} {Theorem} [section]
\newtheorem{lemma} [theorem]{Lemma}
\newtheorem{proposition}[theorem]{Proposition}
\newtheorem{corollary} [theorem]{Corollary}
\theoremstyle{definition}
\newtheorem{definition} [theorem] {Definition}

\newtheorem{remark} [theorem]  {Remark}

\numberwithin{equation}{section}

\title{Weakly 1-completeness of holomorphic fiber bundles over compact K\"ahler manifolds}
\date\today

\author{Aeryeong Seo}
\address{Department of Mathematics, 
Kyungpook National University, 
Daegu 41566, Republic of Korea}%
\email{aeryeong.seo@knu.ac.kr}

\begin{document}

\maketitle
\def\Label#1{\label{#1}{\bf (#1)}~}
\markboth{Aeryeong Seo}{Weakly 1-completeness of holomorphic fiber bundles over compact K\"ahler manifolds}

\begin{abstract}
In \cite{Diederich_Ohsawa_1985} Diederich and Ohsawa proved that every disc bundle over 
a compact K\"ahler manifold is weakly 1-complete. In this paper, under certain conditions we generalize this result
to the case of fiber bundles over compact K\"ahler manifolds whose fibers are bounded symmetric domains. 
In particular if the representation related to 
the fiber bundle is reductive, it has a plurisubharmonic exhaustion function. 
If the bundle is obtained by the diagonal action the product of
on bounded symmetric domains, 
we show that it is hyperconvex. 
\end{abstract}

\section{Introduction}
For a complex manifold we say that it is {\it weakly 1-complete} (or pseudoconvex)
if it admits a $C^\infty$ smooth plurisubharmonic (psh, for short) exhaustion function. 
Two extreme examples of this concept are compact complex manifolds and Stein manifolds; every constant function is a $C^\infty$ smooth psh exhaustion function on a compact complex manifold 
and one of the equivalent conditions to be a Stein manifold is the existence of a
$C^\infty$ smooth strictly psh exhaustion function (\cite{Grauert}).
Weakly 1-complete complex manifolds sit somewhere between them.
For the properties of weakly 1-complete manifolds, see \cite{Ohsawa_1981, Ohsawa_1982, Napier_Ramachandran_1997, Ohsawa_2018} and the references therein.
Recently Mongodi-Slodkowski-Tomassini classified weakly 1-complete complex surfaces 
for which such an exhaustion function can be chosen to be real analytic (see \cite{Mongodi_Slodkowski_Tomassini_2018}).

In this paper we investigate the weakly 1-completeness of fiber bundles whose fibers
are bounded symmetric domains of higher rank.
Let $M$ be a compact K\"ahler manifold and $\pi_1(M)$ be its fundamental group. 
Let $\Omega$ be a bounded symmetric domain and $E\rightarrow M$
a holomorphic fiber bundle over $M$ with fiber $\Omega$.
One can express $E$ as $M\times_\rho \Omega$, where 
$\rho$ is a homomorphism from $\pi_1(M)$ to the set of automorphisms
of $\Omega$, denoted by $\text{Aut}(\Omega)$,
and $\pi_1(M)$ acts on the universal cover 
$\widetilde M$ of $M$ as the deck transformation.
We will stick to the notation introduced above through the whole paper.
The main result is the following:

\begin{theorem}\label{main theorem1}
Let $E =  M \times_\rho \Omega$ be a holomorphic fiber bundle over a compact K\"ahler manifold $M$
with a bounded symmetric domain fiber $\Omega$ where $\rho\colon \pi_1(M)\rightarrow \text{Aut}(\Omega)$ is a homomorphism.
Suppose that $\rho$ is reductive in $\text{Aut}(\Omega)$.
Then $E$ is weakly 1-complete.
\end{theorem}
When $\rho$ is non-reductive, under some restrictive conditions, the fiber bundle is weakly 1-complete (see Theorem \ref{main theorem2} and Theorem \label{theorem_ball}). 
In particular  any $\mathbb B^2$ fiber bundle over a compact K\"ahler manifold is weakly 1-complete
where $\mathbb B^2 = \{ z\in \mathbb C^2 : |z|<1\}$ is the two dimensional unit ball.

One of the most celebrated theorem in several complex variables is a theorem of Oka-Grauert
which says that every locally pseudoconvex domain in a Stein manifold is Stein. 
In this theorem the Steinness of the ambient space is crucial since in 1960's Grauert found an 
example which is locally pseudoconvex but not Stein and later Diederich-Forn\ae ss  constructed a locally trivial holomorphic disc bundle over a Hopf manifold which is locally 
pseudoconvex inside a $\mathbb P^1$-bundle but it does not admit a psh exhaustion function (\cite{Diederich_Fornaess}).
More studies along this line were done in \cite{Ohsawa_2015, Deng_Fornaess}.
Interestingly, Diederich-Ohsawa proved that the non-K\"ahlerness of the base manifold 
is necessary. More precisely they showed the following result.
\begin{theorem}[Diederich-Ohsawa  \cite{Diederich_Ohsawa_1985}]
Every holomorphic disc bundle over a 
compact K\"ahler manifold is weakly 1-complete.
\end{theorem}

The proof of this result consists of two cases corresponding to whether there exist harmonic sections.
At first they consider the function $\phi(z,w) := 1- \left| \frac{w-z}{\overline w z-1} \right|^2$ on the bidisc $\Delta^2$ which is invariant with respect to the diagonal action 
$(z,w)\mapsto (\gamma z, \gamma w)$ for all $\gamma \in \text{Aut}(\Delta)$.
If there exists a harmonic section $h$, then the function $\varphi(z,w) := -\log\phi(h(z),w)$ defines 
a psh exhaustion function.
If there is no harmonic section, they look at the $\mathbb P^1$-bundle 
$ M\times_\rho \mathbb P^1$ which contains the disc bundle $ M\times_\rho \Delta$ as an open subset.
By Eells-Sampson (\cite{Eells_Sampson_1964}) and Hamilton (\cite{Hamilton}),
there exists a flat section $s$ to the ambient bundle
$ M\times_\rho \mathbb P^1$.
One can deduce that the complement of $s(M)$ 
in $ M\times_\rho \mathbb P^1$
has the structure of locally trivial holomorphic $\mathbb C$-bundle with respect to the group 
$\left\{z\mapsto az+b : |a|=1, a,b\in \mathbb C\right\}$.
Here they used the fact that $s(M)$ is contained in a real hypersurface in $M \times_\rho \mathbb P^1$ and as a consequence 
the normal bundle of $s(M)$ in $ M\times_\rho \mathbb P^1$ is a topologically trivial line bundle.
Knowing this last fact, they could construct a psh exhaustion function.

The study of bounded symmetric domain fiber bundles reduces to two cases according to
whether $\rho$ is reductive. 
Note that for a real reductive Lie group $G$ with its Lie algebra $\frak g$, a representation $\rho\colon \pi_1(M)\rightarrow G$ is said to be reductive if $Ad\circ \rho\colon \pi_1(M)\rightarrow \text{Aut}(\frak g)$ is completely reducible.
If $G$ is algebraic, this definition is equivalent to that the Zariski closure of $\rho(\pi_1(M))$ in $G$ is reductive.

First we consider a real-valued function on $\Omega \times \Omega$ defined by 
\begin{equation}\nonumber
\psi_\Omega (z,w) := \frac{K_\Omega(z,z)K_\Omega(w,w)}{|K_\Omega(z,w)|^2},
\end{equation}
which is invariant under the diagonal action of $\text{Aut}(M)$.
Here $K_\Omega$ denotes the Bergman kernel of $\Omega$.
When $\Omega$ is the unit disc, it coincides with $1 / \phi$. In Section \ref{psh}, we show that when $\Omega$ is a bounded symmetric domain, 
$\log \psi_\Omega$ is psh on $\Omega\times \Omega$ by a straight forward calculation which uses an explicit formula of the Bergman
kernel of $\Omega$.
We then apply the following theorem of Corlette:
\begin{theorem}[Corlette \cite{Corlette_1988} cf. \cite{Toledo_1999} Theorem 4.7]\label{Corlette}
Let $M$ be a compact Riemannian manifold, let $G$ be a semisimple algebraic group,
and let $\rho\colon \pi_1(M)\rightarrow G$ be a representation.
Then a $\rho$-equivariant harmonic map $f\colon \widetilde M\rightarrow G/K$ 
exists if and only if $\rho$ is reductive.
\end{theorem}

When $\rho$ is reductive, we can show that the function $\log \psi_\Omega(h(z),w)$ is a psh exhaustion function where $h$ is the harmonic section obtained in Theorem \ref{Corlette}.
The fact that $h$ is pluriharmonic, which was proved by Siu \cite{Siu_1980} and Sampson \cite{Sampson_1986} (see also \cite{Toledo_1999}), is critically used in the proof.

If $\rho$ is non-reductive, there exists no harmonic section, which implies that there exists a harmonic section from $M$ to $M\times_\rho B$ where $B$ is a boundary component of $\Omega$ in its compact dual $\widehat \Omega$ by Hamilton (\cite{Hamilton}).
Now exploiting a generalized Cayley transformation of Wolf-Kor\'anyi \cite{Koranyi_Wolf_1965}, we could obtain some results in some special situations (Theorem \ref{main theorem2}, Theorem \ref{theorem_ball}).

Let us recall that a complex manifold is hyperconvex when it admits a bounded psh exhaustion function.
Here we say that $\mu \colon X\rightarrow (-\infty, 0]$ on a manifold $X$ is a bounded exhaustion function
if $\{p\in X : \mu(p) < c\}$ is relatively compact in $X$ for all negative $c\in \mathbb R$.
It is important to stress that any hyperconvex complex manifold is weakly 1-complete but the converse is not true.
For the counterexample, see \cite{Ohsawa_Sibony}.
Let $\Gamma\subset  \text{Aut}(\Omega)$ be a cocompact discrete subgroup of  $\text{Aut}(\Omega)$.
Then $\Gamma\setminus\Omega$ is a compact K\"ahler manifold with the metric 
induced from the Bergman metric on $\Omega$.  Consider the diagonal action of $\Gamma$ 
on $\Omega\times\Omega$, i.e. $\gamma(z,w) = (\gamma z, \gamma w)$.
Then the quotient complex manifold of $\Omega\times \Omega$ by the diagonal action, 
$\Omega \times \Omega /\Gamma$, is a $\Omega$-fiber bundle 
over $\Gamma\setminus\Omega$.
In this special case we can show that $\Omega \times \Omega /\Gamma$ is hyperconvex (Theorem \ref{hyperconvex}).
If $\Omega$ is the ball $\mathbb B^n$, one sees that 
$\mathbb B^n\times \mathbb B^n/\Gamma$ is a domain in 
$\mathbb C\mathbb P^n\times \mathbb B^n/\Gamma$ 
with real analytic
boundary whose Diederich-Fornaess index equals to $1 /2$ (Corollary \ref{1/2}).
In case of $n=1$, it was proved by Adachi-Brinkschulte \cite{Adachi_Brinkschulte_2015}
and Fu-Shaw \cite{Fu_Shaw_2016}.
We exploit their theorems to obtain the Diederich-Fornaess index.

\medskip

{\bf Acknowledgement} 
The author would like to thank Masanori Adachi and Takeo Ohsawa for
their helpful comments on preliminary drafts. 
The author also expresses his gratitude to the anonymous referee for many valuable comments.
This work began at the workshop ``Progress in Several Complex Variables" 
which was held in Korea Institute of Advanced Study (KIAS), April 23--27, 2018.
The author is grateful to the organizers
and the Institute for providing a stimulating working environment.
The author was partially supported by Basic Science Research Program through the National Research Foundation
of Korea (NRF) funded by the Ministry of Education (NRF-2019R1F1A1060175).

\section{Preliminaries}
Let $X$ be an irreducible Hermitian symmetric space of non-compact type.
Let $G$ be the identity component of the isometry group of $X$
with respect to the Bergman metric of $X$
and $K\subset G$ the isotropy subgroup at $o\in X$.
Then $X$ is biholomorphic to $G/K$.
Denote by $\mathfrak g$ and by $\mathfrak k$ the Lie algebras of $G$ and $K$ respectively.
Let $\mathfrak g = \mathfrak k + \mathfrak m$ be the Cartan decomposition.
Let $\mathfrak g^\mathbb C = \mathfrak g\otimes_{\mathbb R} \mathbb C$, 
$\mathfrak k^\mathbb C = \mathfrak k \otimes_{\mathbb R} \mathbb C$,
$\mathfrak m^\mathbb C = \mathfrak m\otimes_{\mathbb R} \mathbb C$, 
and $G^\mathbb C$ be the complex Lie group corresponding to $\frak g^\mathbb C$.
Let $\mathfrak g_c = \mathfrak k + \sqrt{-1}  \mathfrak m$
be a Lie algebra of compact type and 
$G_c$ the corresponding connected Lie group of $\mathfrak g_c$.
Then $\widehat X=G_c/K$ is the compact dual of $X$. 
Let $\mathfrak h$ be a Cartan subalgebra of $\mathfrak g$ contained in $\mathfrak k$.
Note that $\mathfrak h^\mathbb C = \mathfrak h\otimes_{\mathbb R} {\mathbb C}$ is a 
Cartan subalgebra of $\mathfrak g^\mathbb C$.
Let $\Psi$ denote the set of roots of $\mathfrak g^\mathbb C$ with respect to $\mathfrak h^\mathbb C$
and let $\mathfrak g^\alpha$ denote the root space with respect to a root $\alpha\in \Psi$.
Let $\Psi_{\mathfrak k}$, $\Psi_{\mathfrak m}$ denote 
the set of compact, non-compact roots of $\mathfrak g^\mathbb C$
with respect to the Cartan decomposition $\mathfrak g^\mathbb C = \mathfrak k^\mathbb C+ \mathfrak m^\mathbb C$
respectively
and choose an order of $\Psi$ such that the set of positive non-compact roots 
$\Psi_{\mathfrak m}^+$ satisfies that 
$\frak m^+:=\sum_{\alpha\in \Psi_{\mathfrak m}^+} \mathfrak g^\alpha = T^{1,0}_o X$.
Here $T^{1,0}X$ denotes the holomorphic tangent bundle of $X$.
Denote $\frak m^-:=\sum_{\alpha\in \Psi_{\mathfrak m}^-} \mathfrak g^\alpha$.
Let $M^+$ and $M^-$ be the corresponding analytic subgroups in $G^\mathbb C$.
Note that $\frak m^+$ and $\frak m^-$ are abelian subalgebras of $\frak g^\mathbb C$.
The center $\frak z$ of $\frak k$ contains an element $Z$ such that $\text{Ad}_Z E = \pm iE$ for 
$ E\in \frak m^{\mp}$. $J:= \text{Ad}_Z$ is a complex structure on $\frak m$.
A basis of $\frak m$ is given by the elements 
$
X_\alpha = E_\alpha+E_{-\alpha}$ and
$
Y_\alpha = -i(E_\alpha - E_{-\alpha})
$
where $\alpha$ is non-compact positive. For such $\alpha$, we have the relations
$
JX_\alpha = Y_\alpha$, $JY_\alpha = -X_\alpha$ and $ [X_\alpha, Y_\alpha] = 2iH_\alpha$.
Define $X_\alpha^c = iX_\alpha$ and $Y_\alpha^c = iY_\alpha$.
Those define a basis of $i\frak m$.
$K^\mathbb C$ denoting the analytic subgroup corresponding to $\frak k^\mathbb C$, 
$K^\mathbb C\cdot M^+$ is a semidirect product.
$\widehat X=G_c/K$ is identified with $G^\mathbb C/K^\mathbb C\cdot M^+$ by the identity map of $G$ into $G^\mathbb C$.
For $o= eK \in \widehat X$, the orbit $G \cdot o$ is the image of the holomorphic embedding $gK\mapsto g(o)$ of $X$ into $\widehat X$ 
(Borel embedding). The map $\xi \colon \frak m^- \rightarrow \widehat X$ defined by 
\begin{equation}\nonumber
\xi(E) = \text{exp}(E)(o)
\end{equation}
is a holomorphic homeomorphism onto a dense open subset and $\xi$ is $\text{Ad}_{K}$-equivariant.
Then $
\Omega = \xi^{-1}(G(o))
$
is a bounded symmetric domain in $\frak m^-$; this is the Harish-Chandra realization of $X$.

For $\alpha, \,\beta\in \Psi$, one says that $\alpha$ and $\beta$ are strongly orthogonal
if and only if $\alpha \pm \beta \notin \Psi$.
Let $\Pi:= \{\alpha_1, \ldots, \alpha_r\}$ denote a maximal set of strongly orthogonal 
positive non-compact roots of $\frak g^\mathbb C$.
Then $X$ is of rank $r$.
For every $\alpha\in \Pi$ we have 
\begin{equation}\nonumber
c_\alpha := \text{exp}\left(\frac{\pi}{4}X_\alpha \right)\in G_c, \quad c = \prod_{\alpha\in \Pi} c_\alpha
\end{equation}
and $c$ is called the Cayley transformation of $X$. Note that $c$ has order 4 or 8. 
For $\Lambda\subset \Pi$ partial Cayley transformation is defined by
\begin{equation}\label{partial cayley}
c_\Lambda = \prod_{\alpha\in \Lambda} c_\alpha.
\end{equation}

Denote by $\mathfrak g_\Lambda^\mathbb C$ the derived algebra of 
$\mathfrak h + \sum_{\alpha \perp \Pi\setminus \Lambda}\mathfrak g^\alpha$,
where $\perp$ is the orthogonality with respect to the inner product 
induced by the Killing form of $\mathfrak g^\mathbb C$.
Let $G_\Lambda^\mathbb C$ denote the Lie subgroup of $G^\mathbb C$ corresponding to $\mathfrak g_\Lambda^\mathbb C$
and $G_{\Lambda}$, $G_{\Lambda}^c$, $\frak g_{\Lambda}$, $\frak g_\Lambda^c$ 
denote $G\cap G_\Lambda^\mathbb C$, $G^c\cap G_\Lambda^\mathbb C$,
$\frak g\cap \frak g_\Lambda^\mathbb C$, $\frak g^c\cap \frak g_\Lambda^\mathbb C$.
Let $X^c_\Lambda = G^c_\Lambda \cdot o$ and 
$X_{\Lambda,0} = G_{\Lambda}\cdot o \subset X$.

Let $\partial X$ be the topological boundary of $X$ in $\widehat X$ and $\Delta=\{z\in \mathbb C : |z|<1\}$ the unit disc. A holomorphic map $g\colon \Delta \rightarrow \widehat X$ such that $g(\Delta)\subset \partial X $ is called a holomorphic arc in $\partial X$. A finite sequence $\{g_1, \ldots, g_s\}$ of holomorphic arcs in $\partial X$ is called a chain of holomorphic arcs in $\partial X$ if $f_j(\Delta)\cap f_{j+1}(\Delta) \neq \emptyset$ for any $j = 1,\ldots, s-1$. One can give an equivalence class on $\partial X$ such that for $z_1,\,z_2\in \partial X$, $z_1\sim z_2$ if and only if there is a chain of holomorphic arcs $\{g_1,\ldots, g_s\}$ in $\partial X$ with $z_1\in g_1(\Delta)$ and $z_2\in g_s(\Delta)$. The equivalence classes are the boundary components of $\partial X$ in $\widehat X$.

\begin{theorem}[Wolf \cite{Wolf_1972}]
The $G$ orbits on the topological boundary of $X$ in its compact dual
are the sets 
\begin{equation}\nonumber
G(c_{\Pi-\Lambda}o) = \bigcup_{k\in K}kc_{\Pi-\Lambda}X_{\Lambda,0}
\quad \text{ with }\quad \Lambda \subsetneqq \Pi,
\end{equation}
where $c_{\Pi-\Lambda}$ is the Cayley transformation with respect to $\Pi-\Lambda$.
Furthermore the boundary components of $X$ in $\widehat X$ are the sets 
$kc_{\Pi-\Lambda}X_{\Lambda, 0}$ with $k\in K$ and $\Lambda \subsetneqq \Pi$.
These are Hermitian symmetric spaces of non-compact type and rank is $|\Lambda|$.
\end{theorem}



\begin{definition}\label{k-component}
If $|\Lambda|  = |\Pi|-a$, then we will call $kc_{\Pi-\Lambda}X_{\Lambda, 0}$ 
an {\it $a$-component}. 
\end{definition}

Remark that for any $1$-component $kc_{\Pi-\Lambda}X_{\Lambda,0}$,
there exists a unit disc $\Delta$ so that $\Delta\times kX_{\Lambda,0}$ can be 
embedded totally geodesically in $X$ and $kc_{\Pi-\Lambda,0}X_{\Lambda,0}\cong \{e^{i\theta}\}\times kX_{\Lambda,0}$
for some $e^{i\theta}\in \partial\Delta$ (cf. Proposition 1.7 in \cite{Mok_Tsai_1992}).

Let $B^\Lambda$ be the set of all elements of $G$ which preserves 
$c_{\Pi-\Lambda} X_{\Lambda}$ and $\frak b^\Lambda$ its Lie algebra.
By Kor\'anyi and Wolf, $B^\Lambda$ and $\frak b^\Lambda$ has the following structure.
Refer \cite{Koranyi_Wolf_1965} for details. We will follow their notation.

\begin{theorem}[Kor\'anyi-Wolf \cite{Koranyi_Wolf_1965}] 
$B^\Lambda$ is a maximal parabolic subgroup of $G$, and is the normalizer of $\text{Ad}_{c_{\Pi-\Lambda}^{-1}} N^{\Lambda,-}$ in $G$. The identity component of $B^\Lambda$ is given 
by 
\begin{equation}\label{decomposition}
\left\{G_{\Lambda}\cdot L_2^\Lambda \cdot 
\text{Ad}_{c_{\Pi-\Lambda}^{-1}} K_{\Pi-\Lambda, 1}^*\right\}\cdot \text{Ad}_{c_{\Pi-\Lambda}^{-1}} N^{\Lambda,-}
\end{equation}
and this is the Chevalley decomposition into reductive and unipotent parts.
\end{theorem}

Note that we have \begin{equation}\nonumber
\frak m^- = \frak m_{\Pi-\Lambda,1}^- + \frak m_2^{\Lambda, -} + \frak m_\Lambda^-
\end{equation}
and for $E\in \frak m^-$, we can express $E = E_1 + E_2 + E_3$ with 
$E_1\in  \frak m_{\Pi-\Lambda,1}^- $, $E_2 \in \frak m_2^{\Lambda, -}$ and 
$E_3 \in  \frak m_\Lambda^-$.
Let $\nu$ be a complex antilinear map of $\frak m^\pm$ onto $\frak m^\mp$ preserving this direct decomposition.
For all $W\in \Omega \cap \frak m_\Lambda^-$, define the linear transformation 
$\mu(W) \colon \frak m_2^{\Lambda, -}\rightarrow \frak m_2^{\Lambda,-}$ by 
$
\mu(W) U = \text{Ad}_W \tau_{\Pi-\Lambda} \nu(U)$ with 
$\tau_{\Pi-\Lambda} = \text{Ad}_{c_{\Pi-\Lambda}^2}$.
For all $V\in \frak m_2^{\Lambda,-}$, define the linear function 
$f_V\colon \frak m_\Lambda^-\rightarrow \frak m_2^{\Lambda,-}$ by 
$f_V(W) = (I+\mu(W))V$.
Finally, for all $W\in \Omega \cap \frak m_\Lambda^-$, define the vector-valued bilinear form 
$F_W\colon \frak m_2^{\Lambda, -}\times \frak m_2^{\Lambda,-} \rightarrow \frak m_{\Pi-\Lambda,1}^-$ by 
$
F_W(U,V) = -\frac{i}{2} \left[ U, \tau_{\Pi-\Lambda}(\nu (I+\mu(W)))^{-1}V\right]$.

\begin{theorem}[Kor\'anyi-Wolf \cite{Koranyi_Wolf_1965}] \label{Cayley transform}
\leavevmode
\begin{enumerate}
\item $L_2^\Lambda \cdot 
\text{Ad}_{c_{\Pi-\Lambda}^{-1}} K_{\Pi-\Lambda, 1}^*\cdot \text{Ad}_{c_{\Pi-\Lambda}^{-1}} N^{\Lambda,-}$ acts on $B^\Lambda$ trivially.
\item \label{translation}
$N^{\Lambda,-}$ acts on $\frak m^-$ by
\begin{equation}\label{translate}
g(E) = E+U+f_V(E_3) + 2iF_{E_3}\left(E_2, f_V(E_3)\right) + iF_{E_3}\left(f_V(E_3), f_V(E_3)\right)
\end{equation}
where $g=\text{exp}\left(U+(I-\tau_{\Pi-\Lambda})(V)\right)$, 
$U\in \frak n_1^{\Lambda,-}$ and $V\in \frak m_2^{\Lambda,-}$. 
\item \label{adjoint}
$K^{\Lambda*}$ acts on $\frak m^-$ by the adjoint representation and it preserves $\frak m_{\Pi-\Lambda,1}^-$, $\frak m_2^{\Lambda,-}$ and $\frak m_\Lambda^-$.
\item \label{real}
On $\frak m_{\Pi-\Lambda,1}^-$,  $K_{\Pi-\Lambda,1}^*$ is real, {$G_\Lambda$} and $L_2^{\Lambda}$ are trivial. These actions are $\xi$-equivariant; in particular $K^{\Lambda *}\cdot N^{\Lambda,-}$ preserves $\xi(\frak m^-)$.
\item \label{K}
 For all $k\in K^{\Gamma *}$, $W\in \Omega_\Gamma$ and $U, V\in \frak m_2^{\Gamma,-}$, we have 
\begin{equation}\label{bilinear form}
 \text{Ad}_k F_W(U,V) = F_{\text{Ad}_k W} \left(\text{Ad}_k U, \text{Ad}_k V\right).
\end{equation}
\item \label{c(Omega)}
$\Omega$ can be realized as a homogeneous Siegel domain of the third type given by
\begin{equation}\nonumber
c_{\Pi-\Lambda} \Omega = \left\{E\in \frak m^- \colon \text{Im } E_1 - \text{Re } F_{E_3}(E_2,E_2)\in c^\Lambda,\,
E_3\in X_{\Lambda,0} =\Omega\cap \frak m_\Lambda^- \right\}.
\end{equation}
\end{enumerate}
\end{theorem}
\begin{lemma}\label{bracket} 
$[[\frak m_\Lambda^+, \frak m_2^{\Lambda,-}], \frak m_\Lambda^+]=0$
\end{lemma}
\begin{proof}
By Lemma 6.3 in \cite{Koranyi_Wolf_1965}, one sees that $\text{Ad}_{c_{\Pi - \Lambda}^{-1}} \frak r_2^{\Lambda,-}$
is an eigenspace with an eigenvalue $-1$ of $\text{ad}_{-Y_{\Pi-\Lambda}}$ where  
$\frak r_2^{\Lambda,-} = \frak q_2^{\Lambda, +} + \frak m_2^{\Lambda,-}$ with 
$\frak q_2^{\Lambda, +}\subset \frak k^\mathbb C$.
Moreover $\frak m^+_\Lambda$ belongs to the $0$-eigenspace of $\text{ad}_{-Y_{\Pi-\Lambda}}$.
Hence $$[\frak m_\Lambda^+, \frak m_2^{\Lambda,-}]\subset (\frak q_2^{\Lambda, +} + \frak m_2^{\Lambda,-})\cap \frak k^\mathbb C= \frak q_2^{\Lambda,+}.$$

Note that 
$[\frak m_\Lambda^+, \frak q_2^{\Lambda,+}]$ belongs to the positive root space.
Since $$[\frak m_\Lambda^+, \frak q_2^{\Lambda,+}]\subset \frak r_2^{\Lambda, -}\cap
\frak m^\mathbb C=\frak m_2^{\Lambda,-},$$ it should be zero.

\end{proof}

\section{Pluriharmonicity of invariant functions}\label{psh}
In this section, we denote by $M^{\mathbb C}_{p,q}$ the set of $p\times q$ complex matrices.
Denote by $SM^{\mathbb C}_{n,n}$ (resp. $ASM^{\mathbb C}_{n,n}$) the set of symmetric (resp. antisymmetric) 
$n\times n$ complex matrices.
Irreducible bounded symmetric domains consist of the following 
four kinds of classical type and two exceptional type domains:

\begin{enumerate}[(I)]
\item $\Omega_{p,q}^I = \left\{ Z\in M_{p,q}^{\mathbb C} : I_r -  ZZ^* >0 \right\}$,\label{1}
\item $\Omega_{n}^{II} = \left\{ Z\in  M_{n,n}^{\mathbb C} : I_n - ZZ^*>0, \,\, Z^t = -Z  \right\}$, \label{2}
\item $\Omega_{n}^{III} = \left\{ Z\in M_{n,n}^{\mathbb C} : I_n - ZZ^*  >0,\,\, Z^t = Z  \right\}$,\label{3}
\item $ \Omega_n^{IV} = \left\{ Z=(z_1,\ldots, z_n)\in {\mathbb C}^n : 
ZZ^* <1 ,\, 0< 1-2 ZZ^*+ \left| ZZ^t \right|^2   \right\}$, \label{4}
\item $ \Omega_{16}^V = \left\{z\in M^{\mathbb O_\mathbb C}_{1,2} : 
1-(z|z) + (z^\#|z^\#) >0,\, 2-(z|z)>0 \right\}$, and \label{5}
\item $
\Omega_{27}^{VI} = 
\left\{ z\in H_3(\mathbb O_\mathbb C)  : 
1-(z|z) + (z^\#|z^\#) - |\det z|^2 >0,\, \right.\\
\left.
\quad\quad\quad\quad\quad\quad
\quad\quad\quad\quad\quad\quad
\quad\quad\quad\quad\quad\quad
 3-2(z|z) + (z^\#|z^\#) >0,\, 3-(z|z)>0 \right\}.
$\label{6}
\end{enumerate}
Here $\mathbb O_\mathbb C$ is the complex 8-dimensional algebra of complex octonions.
For $a=(a_0,a_1,\ldots, a_7)\in \mathbb O_\mathbb C$ with $a_i\in \mathbb C$, 
let $a \mapsto \tilde a:=(a_0, -a_1,\ldots, -a_7) $ denote the Cayley conjugation
and $a\mapsto \overline a:= (\overline a_0, \overline a_1, \ldots, \overline a_7)$ 
the complex conjugation.
The Hermitian scalar product is given by $(a|b) = a\tilde{\overline b} + \tilde a\overline b$.
Let $H_3(\mathbb O_{\mathbb C})$ be the complex vector space of $3\times 3$ 
matrices with entries in $\mathbb O_\mathbb C$ 
which are Hermitian with respect to the Cayley conjugation in $\mathbb O_{\mathbb C}$.
Explicitly $A\in H_3(\mathbb O_\mathbb C)$ can be expressed as 
\begin{equation} \label{a}
A = \left( \begin{array}{ccc}
\alpha_1 & a_3 & \tilde a_2 \\
\tilde a_3 & \alpha_2 & a_1 \\
a_2 & \tilde a_1 & \alpha_3
\end{array}\right)
\quad \text{ with } a_i\in \mathbb O_\mathbb C
\text{ and } \alpha_i \in \mathbb C \text{ for all } i=1,2,3.
\end{equation}
For $A\in H_3(\mathbb O_\mathbb C)$ expressed by \eqref{a}, 
let $A^\# \in H_3(\mathbb O_\mathbb C)$
be the adjoint matrix of $A$ expressed by
\begin{equation}\label{explicit}
A^\# = \left( \begin{array}{ccc}
\alpha_2\alpha_3-a_1\tilde a_1 
& \tilde a_2\tilde a_1 - \alpha_3 a_3 
& \widetilde{\tilde a_1\tilde a_3 - \alpha_2 a_2 } \\
\widetilde{ \tilde a_2\tilde a_1 - \alpha_3 a_3} 
& \alpha_3 \alpha_1-a_2\tilde a_2 
& \tilde a_3\tilde a_2 - \alpha_1 a_1 \\
\tilde a_1\tilde a_3 - \alpha_2 a_2 
&  \widetilde{\tilde a_3\tilde a_2 - \alpha_1 a_1}
& \alpha_1\alpha_2-a_3\tilde a_3
\end{array}\right).
\end{equation}
The Hermitian scalar product on $H_3(\mathbb O_\mathbb C)$ is given by 
$$
(A|B) 
= \sum_{i=1}^3 \alpha_i \overline\beta_i + \sum_{i=1}^3 (a_i|b_i).
$$
Explicitly, 
\begin{equation}\nonumber
\begin{aligned}
(A|A) &= \sum_{i=1}^3 |\alpha_i|^2 + 2 \sum_{i=1}^3 \left( |a_{i0}|^2 +\cdots+ |a_{i7}|^2\right)\\
(A^\#|A^\#) &= |\alpha_2\alpha_3-a_1\tilde a_1|^2 + |\alpha_3 \alpha_1-a_2\tilde a_2|^2\
+|\alpha_1\alpha_2-a_3\tilde a_3|^2\\
&\quad +(\tilde a_3\tilde a_2 - \alpha_1 a_1|\tilde a_3\tilde a_2 - \alpha_1 a_1)
+(\tilde a_1\tilde a_3 - \alpha_2 a_2 |\tilde a_1\tilde a_3 - \alpha_2 a_2)
\\
&\quad+(\alpha_1\alpha_2-a_3\tilde a_3|\alpha_1\alpha_2-a_3\tilde a_3)\\
|\det A|^2 &= \left| \alpha_1\alpha_2\alpha_3 - \sum_{i=1}^3 \alpha_i a_i\tilde a_i 
+ a_1(a_2a_3) + (\tilde a_3 \tilde a_2) \tilde a_1\right|^2.
\end{aligned}
\end{equation}
with $a_i=(a_{i0},\cdots, a_{i7})\in \mathbb O_\mathbb C$ for $i=1,2,3$.
Let $M^{\mathbb O_{\mathbb C}}_{1,2}$ denote the set of $1\times 2$ complex octonion matrices.
For $ z=(z_1,z_2)\in M^{\mathbb O_\mathbb C}_{1,2}$, we identify $z$ with 
$$\left( \begin{array}{ccc}
0& z_2&\tilde z_1\\
\tilde z_2 & 0&0\\
z_1&0&0
\end{array}\right) \in H_3(\mathbb O_{\mathbb C})$$
and apply the same notation $\#$, $(\cdot, \cdot)$ and so on. See \cite{Roos_2008} for more details.

\begin{lemma}\label{Bergman kernel of BSD}
Let $\Omega$ be an irreducible bounded symmetric domain and 
$N_\Omega$ its generic norm.
Then the Bergman kernel $K(z,z)$ of $\Omega$ can be
expressed by $c_1N_\Omega(z)^{-c_2}$ for some constant $c_1, c_2>0$ where $N_\Omega$ is the generic norm of $\Omega$. 
\end{lemma}
The generic norms $S_{p,q}^{I}$, $S_n^{II} $, $S_n^{III}$, $S_n^{IV}$, $S^{V}$ and $S^{VI}$ and $c_2$ in Lemma \ref{Bergman kernel of BSD} of the corresponding domains are given by
\smallskip
\begin{enumerate}[(i)]
\item
$S_{p,q}^{I}(Z,\overline Z) = \det(I_r -ZZ^*) 
\quad \text{ for }  Z\in M_{p,q}^{\mathbb C} $, $\quad $$c_2 = p+q$,
\item
$S_n^{II} (Z,\overline Z) = s_n^{II}(Z) 
\quad\text{ for } Z\in ASM^{\mathbb C}_{n,n}$, $\quad c_2 = 2(n-1)$,
\item
$S_n^{III} (Z,\overline Z)= \det(I_n -ZZ^*) 
\quad\text{ for } Z\in SM^{\mathbb C}_{n,n}$, $\quad c_2 = n+1$,
\item
$S_n^{IV}(Z,\overline Z) = 1-2 ZZ^* + \left| ZZ^t \right|^2 
\quad\text{ for } Z\in {\mathbb C}^n$, 
$\quad c_2 = n$,
\item
$S^{V}(Z,\overline Z) = 1-(Z|Z) + (Z^\#|Z^\#) 
\quad \text{ for }Z\in M^{\mathbb O_{\mathbb C}}_{1,2}$, $\quad c_2 =12$ and 
\item
$S^{VI}(Z,\overline Z) = 1-(Z|Z) + (Z^\#|Z^\#) -\left|\det Z\right|^2
\quad  \text{ for } Z\in H_3(\mathbb O_{\mathbb C})$, $\quad c_2 = 18$
\end{enumerate}
with $\det(I_n -ZZ^*) = s_n^{II}(Z)^2$ for some polynomial 
$s_n^{II}(Z)$ and $Z\in ASM^{\mathbb C}_{n,n}$.

\medskip

For a bounded symmetric domain $\Omega$, possibly reducible, 
let $K_\Omega\colon \Omega\times\Omega\rightarrow \mathbb C$ denote the 
Bergman kernel of $\Omega$.
Define a function $\psi_\Omega \colon \Omega\times \Omega\rightarrow \mathbb R$ by
\begin{equation}\label{psi}
\psi_\Omega (z,w) := \frac{K_\Omega(z,z)K_\Omega(w,w)}{|K_\Omega(z,w)|^2}.
\end{equation}
Since  $\psi_\Omega(\gamma(z),\gamma(w)) = \psi_\Omega(z,w)$ for any 
$\gamma\in \text{Aut}(\Omega)$
and $\psi_\Omega(z,w)>1$ whenever $z\neq w$, $\log \psi_\Omega$ is a well defined function on $\Omega\times\Omega$ which is invariant under the diagonal action of $\text{Aut}(\Omega)$ and positive off the diagonal.
\begin{lemma}\label{lemma_psh}
Let $\Omega$ be a bounded symmetric domain, possibly reducible.
$\log \psi_\Omega$ is a $C^\infty$ psh function
which is invariant under the diagonal action of $\text{Aut}(\Omega)$.
\end{lemma}

\begin{proof}
Since $\Omega$ is a finite product on irreducible bounded symmetric domains,
we may assume that $\Omega$ is irreducible.
Since $\psi_\Omega$ is invariant under the diagonal action of $\text{Aut}(\Omega)$, we only
need to prove that $\partial \overline \partial \log \psi_\Omega(z,w) $ is positive semi-definite at $(0,w_0)$, when $w_0$ belongs to the maximal polydisc of $\Omega$.\\
{\bf Type I, $\Omega_{p,q}^I$: }
Since $K_{\Omega_{p,q}^I}(z,w) = c_1\det( I - z\overline w^t)^{-c_2}$ for some constants $c_1, c_2 >0$, we have
 \begin{equation}\label{ddlogpsi}
\begin{aligned}
\partial \overline \partial \log \psi_{\Omega_{p,q}^I}(z,w) 
&= \partial \overline \partial \left( \log K_{\Omega_{p,q}^I}(z,z) + \log K_{\Omega_{p,q}^I}(w,w)
-  \log K_{\Omega_{p,q}^I}(z,w) -  \log K_{\Omega_{p,q}^I}(w,z)\right)\\
&= -c_2 \partial \overline \partial \left(\log \det( I - z\overline z^t) + \log \det( I - w\overline w^t\right) 
- \log \det( I - z\overline w^t) - \log \det( I - w\overline z^t))\\
&= c_2 \left( \begin{array}{cc}
-\partial_z \overline \partial_z \log \det( I - z\overline z^t) 
& \partial_z \overline \partial_w \log \det( I - z\overline w^t) \\
\partial_w \overline \partial_z \log \det( I - w\overline z^t) 
&-\partial_w \overline \partial_w \log \det( I - w\overline w^t) 
\end{array}\right)\\
&=c_2 \left( \begin{array}{cc}
I&-I \\
-I&-\partial_w \overline \partial_w \log \det( I - w\overline w^t) 
\end{array}\right)\quad \text{at } (0,w_0).
\end{aligned}  
\end{equation}
If $-\partial_w\overline \partial_w \log \det (I-w\overline w^t) \geq I$, 
for $X,\, Y \in \mathbb C^{pq}$ we have
\begin{equation}\nonumber
\begin{aligned}
&(X,Y) \left( \begin{array}{cc}
I&-I \\
-I&-\partial_w \overline \partial_w \log \det( I - w\overline w^t) 
\end{array}\right) 
\left( \begin{array}{c}
\overline X^t\\
\overline Y^t
\end{array}\right)\\
&=\, X\overline X^t - Y\overline X^t - X\overline Y^t - Y \partial_w\overline \partial_w 
\log \det (I-w\overline w^t)\overline Y^t
\geq\, (X-Y) \overline{(X-Y)}^t
\end{aligned}
\end{equation}
and hence $\partial\bar\partial \log \psi_{\Omega_{p,q}^I}$ is plurisubharmonic.
Note that \begin{equation}\label{logdet}
\begin{aligned}
- \frac{\partial^2\log \det(I-w\overline w^t)}{\partial w_{ij} \partial\overline w_{kl}}
&=\frac{-1}{\det(I-w\overline w^t)} 
\frac{\partial^2\det(I-w\overline w^t)}{\partial w_{ij}  \partial\overline w_{kl}}\\
  &\quad\quad+\frac{1}{\det(I-w\overline w^t)^2} 
\frac{\partial\det(I-w\overline w^t)}{\partial w_{ij}} 
\frac{\partial\det(I-w\overline w^t)}{\partial\overline  w_{kl}}.
\end{aligned}
\end{equation}
Let $w = (w_{ij})\in M_{p,q}^\mathbb C$ with $p\leq q$. 
Denote $w=\left( \begin{array}{c}
w_1\\
\vdots\\
w_p
\end{array}\right)$ with $w_j = (w_{j1}, \ldots, w_{jq})$.
Let's consider $\partial\overline \partial \det( I - w\overline w^t)$ at
$$w_0=\text{diag}(w_{11},\ldots, w_{pp}):= \left(\begin{array}{ccc|c}
w_{11}&0&0&0\\
0&\ddots&0&0\\
0&0& w_{pp}&0
\end{array} \right).$$
Since we have
\begin{equation}\label{I-ww}
I-w\overline w^t = \left( \begin{array}{cccc}
1-w_1\overline w_1^t& -w_1\overline w_2^t& \ldots& -w_1\overline w_p^t\\
-w_2\overline w_1^t&1 -w_2\overline w_2^t& \ldots& -w_2\overline w_p^t\\
\vdots  &\vdots &&\vdots  \\
 -w_p\overline w_1^t& -w_p\overline w_2^t& \ldots& 1-w_p\overline w_p^t
\end{array}
\right)
\end{equation}
by the derivative formula of the determinant, one obtains 
 \begin{equation}\label{ij}
\frac{\partial\det\left(I-w\overline w^t\right) }{\partial w_{ij}} 
= -\delta_{ij}\frac{  \overline w_{ii}\Pi_{\sigma=1}^p \left(1-|w_{\sigma \sigma}|^2\right)}{ 1-|w_{ii}|^2}
\end{equation}
and 
\begin{equation}\nonumber
\begin{aligned}
&\frac{\partial ^2\det(I-w\overline w^t)}{\partial w_{ij} \partial\overline  w_{kl}} 
= \det~ \begin{blockarray}{ccccccc}
&&&k^{th}&&\\
\begin{block}{(ccc|c|cc)c}
1-w_1\overline w_1^t& -w_1\overline w_2^t& \ldots&-w_1\frac{\partial \overline w_k^t}{\partial\overline  w_{kl}}&\ldots& -w_1\overline w_p^t&\\
\vdots  &\vdots &&\vdots&&\vdots & \\\cline{1-6}
-\frac{\partial w_i}{\partial w_{ij}} \overline w_1^t&- \frac{\partial w_i}{\partial w_{ij}} \overline w_2^t
&\cdots& -\frac{\partial w_i}{\partial w_{ij}} \frac{\partial \overline w_k^t}{\partial\overline  w_{kl}}&\ldots& -\frac{\partial w_i}{\partial w_{ij}}\overline  w_p^t &i^{th} \\\cline{1-6}
\vdots  &\vdots &&\vdots&&\vdots&  \\
 -w_p\overline w_1^t& -w_p\overline w_2^t& \ldots& -w_p\frac{\partial\overline  w_k^t}{\partial\overline  w_{kl}}&\ldots& 1-w_p\overline w_p^t&\\
\end{block}
\end{blockarray}.
\end{aligned}
\end{equation}
Since we have 
\begin{equation}\nonumber
\frac{\partial w_i}{\partial w_{ij}} \overline w_\sigma^t 
= \overline w_{\sigma j},\quad
\frac{\partial w_i}{\partial w_{ij}} 
\frac{\partial \overline w_k^t}{\partial \overline w_{kl}}
=\delta_{jl}
\quad \text{ and }\quad
w_\sigma\frac{\partial \overline w_k^t}{ \partial\overline w_{kl}} 
= w_{\sigma l}
\end{equation}
for $\sigma = 1,\ldots, p$, 
one obtains
\begin{equation}\label{matrix1}
\frac{\partial ^2\det(I-w\overline w^t)}{\partial w_{ij} \partial\overline  w_{kl}} 
= \det~ \begin{blockarray}{ccccccc}
&&&k^{th}&&\\
\begin{block}{(ccc|c|cc)c}
1-w_1\overline w_1^t& -w_1\overline w_2^t& \ldots&-w_{1l}&\ldots& -w_1\overline w_p^t&\\
\vdots  &\vdots &&\vdots&&\vdots&  \\\cline{1-6}
-\overline w_{1j}&-\overline w_{2j}&\cdots& - \delta_{jl}&\ldots& -\overline w_{pj} &i^{th}\\\cline{1-6}
\vdots  &\vdots &&\vdots&&\vdots & \\
 -w_p\overline w_1^t& -w_p\overline w_2^t& \ldots& -w_{pl}&\ldots& 1-w_p\overline w_p^t&\\
\end{block}
\end{blockarray}.
\end{equation}

If $i <k$, the $k$-th row of the matrix in \eqref{matrix1} is $(0,\ldots, 0,-w_{kl},0,\ldots,0)$
when $w = \text{diag} (w_{11},\ldots, w_{pp})$ where $-w_{kl}$ 
is located in the $k$-th component.
Hence if $k\neq l$, then \eqref{matrix1} should be zero.
By applying the same way, one obtains that 
\begin{equation}\nonumber
\frac{\partial ^2\det (I-w\overline w^t)}{\partial w_{ij} \partial\overline  w_{kl}}  = 0
\quad\quad \text{ if }
i\neq k \text{ and } k\neq l.
\end{equation}
Hence by \eqref{logdet}, \eqref{ij}, \eqref{matrix1}
\begin{equation}\nonumber
\frac{\partial ^2\log\det (I-w\overline w^t)}{\partial w_{ij} \partial\overline  w_{kl}}  = 0
\quad\quad \text{ if }
i\neq k \text{ and } k\neq l.
\end{equation}

Suppose that $i=k=l \neq j$.
Then $i$-th row of \eqref{matrix1} is zero at 
$w_0 = \text{diag} (w_{11},\ldots, w_{pp})$. Hence by  \eqref{logdet}, \eqref{ij}, \eqref{matrix1}
\begin{equation}\nonumber
\frac{\partial ^2 \log\det (I-w\overline w^t)}{\partial w_{ij} \partial\overline  w_{ii}} = 0
\quad\quad \text{ if }
i \neq j.
\end{equation}

Suppose that $i<k$ and $k=l$. Then \eqref{matrix1} 
at $w_0 = \text{diag} (w_{11},\ldots, w_{pp})$ is given by 
\begin{equation}\label{matrix3}
\frac{\partial ^2\det \eqref{I-ww}}{\partial w_{ij} \partial \overline w_{kk}} 
= \det \left( \begin{array}{ccccc}
M_1& 0&0&0&0\\
0&-\overline w_{ij}&0&-\delta_{kj}&0\\
0&0&M_2&0&0\\
0&0&0&-w_{kk}&0\\
0&0&0&0&M_3\\
\end{array}
\right)
\end{equation}
where $M_1= \text{diag}\left(1-|w_{11}|^2,\ldots, 1-|w_{i-1,i-1}|^2\right)$, 
$M_2=\text{diag}\left( 1-|w_{i+1,i+1}|^2,\ldots, 1-|w_{k-1,k-1}|^2\right)$ and 
$M_3= \text{diag}\left(1-|w_{k+1,k+1}|^2,\ldots,  1-|w_{pp}|^2\right)$.
Hence by \eqref{matrix3}
\begin{equation}\nonumber
\frac{\partial ^2\log\det (I-w\overline w^t)}{\partial w_{ij}  \partial\overline w_{kk}}  = 0
\quad\quad \text{ if }
i\neq k \text{ and } i\neq j.
\end{equation}
 By the similar way one also gets
 \begin{equation}\nonumber
\frac{\partial ^2\det (I-w\overline w^t)}{\partial w_{ii} \partial\overline w_{kk}}  =
\frac{ \overline w_{ii} w_{kk} \Pi_{\sigma=1}^p 
\left(1-|w_{\sigma \sigma}|^2\right)}{ \left(1-|w_{ii}|^2\right)\left(1-|w_{kk}|^2\right)}
 \quad\quad\text{ if }
i\neq k
\end{equation}
and 
\begin{equation}\nonumber
\frac{\partial ^2\det \left(I-w\overline w^t\right) }{\partial w_{ii} \partial\overline w_{ii}} 
= -\frac{  \Pi_{\sigma=1}^p \left(1-|w_{\sigma \sigma}|^2\right)}{ 1-|w_{ii}|^2}.
\end{equation}
Hence we obtain
\begin{equation}\nonumber
\frac{\partial^2\log \det\left(I-w\overline w^t\right)}{\partial w_{ii} \partial \overline w_{kk}}
=\frac{ \overline w_{ii} w_{kk} }{ \left(1-|w_{ii}|^2\right)\left(1-|w_{kk}|^2\right)}
-\frac{  \overline w_{ii}}{ \left(1-|w_{ii}|^2\right)}
\frac{ w_{kk}}{ \left(1-|w_{kk}|^2\right)}=0 \quad
\text{ if } i\neq k
\end{equation}
and \begin{equation}\nonumber
\frac{\partial^2\log \det(I-w\overline w^t)}{\partial w_{ii} \partial\overline w_{ii}}
= -\frac{1}{ 1-|w_{ii}|^2}
-\frac{  \overline w_{ii}}{ 1-|w_{ii}|^2}\frac{ w_{ii}}{1-|w_{ii}|^2} 
=  -\frac{1}{ \left(1-|w_{ii}|^2\right)^2}.
\end{equation}

Suppose that $i=k$ and $j=l\leq p$ but $i\neq j$.
Then
\begin{equation}\nonumber
 \frac{\partial ^2}{\partial w_{ij} \partial\overline w_{ij}} 
\det (I-w\overline w^t) 
 = \det \left(\begin{array}{ccc}
 N_1&N_2&0\\
 \overline N_2^t& -1& N_3\\
 0& \overline N_3^t& N_4
 \end{array}
 \right)= - \frac{ \Pi_{\sigma=1}^p \left(1-|w_{\sigma \sigma}|^2\right)}
{ \left(1-|w_{ii}|^2\right)\left(1-|w_{jj}|^2\right)}
\end{equation}
where $N_1= \text{diag}(1-|w_{11}|^2, \ldots, 1-|w_{i-1,i-1}|^2)$,
 $N_2 = (-w_{1j},\ldots, -w_{i-1,j})^t$,
 $N_3= (-\overline w_{i+1,j}, \ldots, -\overline w_{p,j})$
 and $N_4= \text{diag}(1-|w_{i+1,i+1}|^2, \ldots, 1-|w_{pp}|^2)$.
Hence 
\begin{equation}\nonumber
\frac{\partial ^2 \log \det \left(I-w\overline w^t\right) }
{\partial w_{ij} \partial\overline w_{ij}} 
 = \frac{ -1}{ \left(1-|w_{ii}|^2\right)\left(1-|w_{jj}|^2\right)} 
\quad \text{ if } i\neq j \leq p.
\end{equation}

Suppose that $i=k$ and $j=l\geq p+1$ but $i\neq j$. Then 
\begin{equation}\label{i}
 \frac{\partial ^2\log\det (I-w\overline w^t) }
{\partial w_{ij} \partial\overline w_{ij}} 
= - \frac{1}{ 1-|w_{ii}|^2} \quad \text{ if } i\leq p < j\leq q.
\end{equation}

Suppose that $i=k$, $j\neq l$ but $j\neq i$, $l \neq i$.
If $j<i$ and $l<i$, we have 
\begin{equation}\label{i=k1}
\begin{aligned}
\frac{\partial^2 \log\det ( I - w\overline w^t)}
{\partial w_{ij} \partial \overline w_{il}} 
&= \frac{1}{\det(I-w\overline w^t)} \det \left( \begin{array}{ccc}
L_1& L_2& 0 \\
L_3&0&0\\
0&0& L_4
\end{array}\right)=0
\end{aligned}
\end{equation}
with $N_1 = \text{diag}(1-|w_{11}|^2, \ldots, 1-|w_{i-1, i-1}|^2)$, 
$L_2 = (0,\ldots, 0 ,-w_{ll}, 0, \ldots, 0)^t$ 
where $-w_{ll}$ is located in the $l$-th component, 
$L_3 = (0,\ldots, 0, -\overline w_{jj}, 0,\ldots, 0)$ 
where $-\overline w_{jj}$ is located in the $j$-th component 
and $L_4 = \text{diag}(1-|w_{i+1,i+1}|^2, \ldots, 1-|w_{pp}|^2 )$.
This holds since $i$-th row and $j$-th row are dependent.

If $i<j \neq l$ or $l<i<j \leq p$ or $j<i<l\leq p$, by the similar way 
we obtain 
\begin{equation}\nonumber
\frac{\partial^2 \log\det ( I - w\overline w^t) }{\partial w_{ij} \partial \overline w_{il}}
=0.
\end{equation}

Hence for $X= \sum_{i,j}X_{ij}\frac{\partial}{\partial w_{ij}}$
\begin{equation}\label{plurisubharmonicity}
\begin{aligned}
&- \sum_{\substack{1\leq i,k \leq p\\1 \leq j,l \leq q}} X_{ij} 
\frac{ \partial^2 \log \det( I-w \overline w^t )}
{\partial w_{ij}\partial \overline w_{kl}} \overline X_{kl}\\
&= \sum_{1\leq i,j \leq p} \frac{ |X_{ij}|^2}
{ \left(1-|w_{ii}|^2\right)\left(1-|w_{jj}|^2\right)}
+ \sum_{\substack{1\leq i\leq p \\p+1\leq j\leq q}} 
 \frac{ |X_{ij} |^2}{ 1-|w_{ii}|^2}
\geq \sum_{\substack{1\leq i\leq p \\1\leq j\leq q}} |X_{ij}|^2
\end{aligned}
\end{equation}
and hence $\log \psi_\Omega$ is psh.

{\bf Type II and III:}
Since $\psi_{\Omega^{II}_n}$, $\psi_{\Omega^{III}_n}$ are the restriction of 
$\psi_{\Omega^I_{n,n}}$, $\frac{1}{2}\psi_{\Omega^I_{n,n}}$ on $\Omega^{II}_n$, $\Omega^{III}_n$ respectively,
$ \log \psi_{\Omega^{II}_{n}}$, $ \log \psi_{\Omega^{III}_{n}}$  are psh.

{\bf Type IV: }
Since 
$
K_{\Omega_n^{IV}}(z,\overline z) = c_1\left(1-2 zz^* + \big| zz^t \big|^2 \right)^{-c_2}
$
for some constant $c_1, c_2>0$,
we have 
\begin{equation}\label{ddlogpsi4}
\begin{aligned}
\partial \overline \partial \log \psi_{\Omega_n^{IV}}(z,w) 
&=c_2 \left( \begin{array}{cc}
2I&-2I \\
-2I&-\partial_w \overline \partial_w \log\left(1-2 w\overline w^t + \big| ww^t \big|^2\right)
\end{array}\right)
\end{aligned}  
\end{equation}
at $(0,w)$.
 A maximal polydisc in $\Omega_n^{IV}$ is given by 
\begin{equation*}
\Delta^2 :=\left\{(w_1,w_2,0,\ldots,0):
w_1=\lambda(\zeta_1 + \zeta_2), w_2=i\lambda(\zeta_1-\zeta_2), |\zeta_1|<1, \, |\zeta_2|<1\right\},
\end{equation*}
with $\lambda^2 = \frac{i}{4}$.
For the case of type $IV$, we denote $\sqrt{-1}$ by $i$.
Note that on $\Delta^2$, 
\begin{equation}\label{note4}
\begin{aligned}
&ww^t = 4\lambda^2\zeta_1\zeta_2= i\zeta_1\zeta_2 ,\quad
w\overline w^t = \frac{1}{2}\left(|\zeta_1|^2+|\zeta_2|^2\right), \\
&1-2 w\overline w^t + \left| ww^t \right|^2=\left(1-|\zeta_1|^2\right)\left(1-|\zeta_2|^2\right),\\
& 4 w_1 w_2 = -\left(\zeta_1^2-\zeta_2^2\right),\\
&4 w_1\overline w_2 = 4\lambda\left(\zeta_1 + \zeta_2\right)
\overline{i\lambda\left(\zeta_1-\zeta_2\right)} 
= -i \left(\zeta_1+\zeta_2\right)\left(\overline \zeta_1 - \overline \zeta_2\right).
\end{aligned}
\end{equation}
Since $
\frac{\partial^2}{\partial w_j\partial\overline w_k}\left(1-2 w\overline w^t + \left| ww^t \right|^2\right)
= -2\delta_{jk} + 4w_j\overline w_k,
$
 we obtain on $\Delta^2$
\begin{equation}\nonumber
\partial \overline\partial \left(1-2 w\overline w^t + \big| ww^t \big|^2\right) = \left(
\begin{array}{cc|ccc}
-2+4|w_1|^2 & 4w_1\overline w_2& &0  \\
4w_2\overline w_1 & -2 + 4|w_2|^2 & &\\\hline
&&-2&&0\\
&0&&\ddots&\\
&&0&&-2
\end{array}\right).
\end{equation}
Since $
\frac{\partial}{\partial w_j} \left(1-2 w\overline w^t + \left| ww^t \right|^2\right)
= -2\left(\overline w_j - w_j\overline{(ww^t)}\right)
$ on $\Delta^2$,
$\partial S^{IV}_n \overline\partial S^{IV}_n$ is 
\begin{equation}\nonumber
\left(\begin{array}{cc|c}
4\left|\overline w_1-w_1\overline{(ww^t)}\right|^2
&4\left(\overline w_1-w_1\overline{(ww^t)}\right)\left(w_2- \overline w_2(ww^t)\right)&0  \\
4\left(\overline w_2-w_2\overline{(ww^t)}\right)\left(w_1- \overline w_1(ww^t)\right)
 & 4\left|\overline w_2-w_2\overline{(ww^t)}\right|^2&  \\\hline
&0&\quad 0\quad
\end{array}\right). 
\end{equation}
Hence by a straightforward calculation we obtain
\begin{equation}\nonumber
\begin{aligned}
&- S^{IV}_n \frac{\partial^2 S^{IV}_n }{\partial w_1 \partial\overline w_2} 
 + \frac{\partial  S^{IV}_n }{\partial w_1}\frac{\partial  S^{IV}_n }{\partial\overline w_2} 
= i\left(|\zeta_1|^2 - |\zeta_2|^2\right)\left(2-|\zeta_1|^2-|\zeta_2|^2\right),\\
&- S^{IV}_n \frac{\partial^2 S^{IV}_n }{\partial w_1 \partial\overline w_1} 
 + \frac{\partial  S^{IV}_n }{\partial w_1}\frac{\partial  S^{IV}_n }{\partial \overline w_1} 
= 2\left(1-|\zeta_1|^2\right)\left(1-|\zeta_2|^2\right)+\left( |\zeta_1|^2- |\zeta_2|^2\right)^2 
\end{aligned}
\end{equation}
and
\begin{equation}\nonumber
- S^{IV}_n \frac{\partial^2 S^{IV}_n }{\partial w_2\partial\overline  w_2} 
 + \frac{\partial  S^{IV}_n }{\partial w_2}\frac{\partial  S^{IV}_n }{\partial\overline w_2} 
 = 2\left(1-|\zeta_1|^2\right)\left(1-|\zeta_2|^2\right)+\left( |\zeta_1|^2- |\zeta_2|^2\right)^2.
\end{equation}
To show that \eqref{ddlogpsi4} is semi-positive definite, we only need to prove
that $-\partial\overline\partial \log S^{IV}_n - 2I$ is positive semi-definite, 
i.e.
\begin{equation}\nonumber
 \left(
\begin{array}{c|ccc}
 M& &0&\\\hline
&\frac{2}{S^{IV}_n}&&0\\
0&&\ddots&\\
&0&&\frac{2}{S^{IV}_n}
\end{array}\right)-2I
\end{equation}
is semi-positive definite where $M$ is given by
{
\begin{equation}\label{ddbar type4}
\left( \begin{array}{cc}
\frac{1}{(S^{IV}_n)^2}\left( 
2(1-|\zeta_1|^2)(1-|\zeta_2|^2)+( |\zeta_1|^2- |\zeta_2|^2)^2
\right) 
&  \frac{i}{(S^{IV}_n)^2}\left(|\zeta_1|^2 - |\zeta_2|^2\right)\left(2-|\zeta_1|^2-|\zeta_2|^2\right)\\
\frac{-i}{(S^{IV}_n)^2}\left(|\zeta_1|^2 - |\zeta_2|^2\right)\left(2-|\zeta_1|^2-|\zeta_2|^2\right)
& \frac{1}{(S^{IV}_n)^2}\left(
2\left(1-|\zeta_1|^2\right)\left(1-|\zeta_2|^2\right)+\left( |\zeta_1|^2- |\zeta_2|^2\right)^2 \right)
\end{array}\right).
\end{equation}
}
\begin{lemma}
$M-\frac{2}{S^{IV}_n} I\geq 0$.
\end{lemma}
\begin{proof}
We only need to check 
{\begin{equation}\label{matrix2}
\left( \begin{array}{cc}
2\left(1-|\zeta_1|^2\right)\left(1-|\zeta_2|^2\right)
+\left( |\zeta_1|^2- |\zeta_2|^2\right)^2 -2S^{IV}_n
& i\left(|\zeta_1|^2 - |\zeta_2|^2\right)\left(2-|\zeta_1|^2-|\zeta_2|^2\right)\\
-i\left(|\zeta_1|^2 - |\zeta_2|^2\right)\left(2-|\zeta_1|^2-|\zeta_2|^2\right)
& 2\left(1-|\zeta_1|^2\right)\left(1-|\zeta_2|^2\right)
+\left( |\zeta_1|^2- |\zeta_2|^2\right)^2 -2S^{IV}_n
\end{array}\right)
\end{equation}} is positive semi-definite.
Since \eqref{matrix2} is equal to 
{\begin{equation}\nonumber
\left( \begin{array}{cc}
\left( |\zeta_1|^2- |\zeta_2|^2\right)^2 
& i\left(|\zeta_1|^2 - |\zeta_2|^2\right)\left(2-|\zeta_1|^2-|\zeta_2|^2\right)\\
-i\left(|\zeta_1|^2 - |\zeta_2|^2\right)\left(2-|\zeta_1|^2-|\zeta_2|^2\right)
& \left( |\zeta_1|^2- |\zeta_2|^2\right)^2
\end{array}\right)
\end{equation}} 
and the determinant of it is 
\begin{equation}\nonumber
(|\zeta_1|^2-|\zeta_2|^2)^4 - (|\zeta_1|^2-|\zeta_2|^2)^2 (2-|\zeta_1|^2-|\zeta_2|^2)^2
=4(|\zeta_1|^2-|\zeta_2|^2)^2 (1-|\zeta_1|^2)(1-|\zeta_2|^2)\geq 0,
\end{equation}
we obtain the lemma.
\end{proof}
Since $\frac{2}{S^{IV}_n} I_{n-2} - 2I_{n-2}$ is positive definite and $M-2I>M-\frac{2}{S^{IV}_n} I>0$, as a result $\partial \overline \partial \log \psi_{\Omega_n^{IV}}(z,w)$ is 
positive semi-definite.

\bigskip
{\bf Type V: }
For $ Z=\left( \begin{array}{ccc}
0& z_2&\tilde z_1\\
\tilde z_2 & 0&0\\
z_1&0&0
\end{array}\right) \in H_3(\mathbb O_\mathbb C)$ for 
$z_1=(z_{10}, z_{11}, \ldots, z_{17})$ and 
$z_2=(z_{20}, z_{21}, \ldots, z_{27})$, we have 
$$S^{V}(Z,\overline Z) = 1-(Z|Z) + (Z^\#|Z^\#) 
= 1-2\sum_{i=1}^2\sum_{j=0}^7 |z_{ij}|^2
+ \sum_{i=1}^2 \left|\sum_{j=0}^7 z_{ij}^2 \right|^2 + 2(z_1z_2 | z_1z_2).$$
One may notice that calculating $\partial\overline\partial \log \psi_{\Omega^V_{16}}(z,w)$ where $z=0$ and $w$ belongs to a maximal polydisc of $\Omega^V_{16}$ is basically the same with calculating $\partial\overline\partial \log \psi_{\Omega^{IV}_{8}}(z,w)$ where $z=0$ and $w$ belongs to a maximal polydisc of $\Omega^{IV}_8$.

\bigskip

\bigskip
{\bf Type VI: } For $Z\in H_3(\mathbb O_\mathbb C)$, we have 
$S^{VI}(Z,\overline Z) = 1-(Z|Z) + (Z^\#|Z^\#) -|\det Z|^2$
for $Z\in H_3(\mathbb O_\mathbb C)$ and one may notice 
that calculating $\partial\overline\partial \log \psi_{\Omega_{27}^{VI}}$ where $z=0$ and $w$ belongs to totally geodesic $\Delta^3$ of $\Omega_{27}^{VI}$ is similar to calculating $\partial \overline\partial \log \psi_{\Omega_{3,3}^I}(0,w)$.
\end{proof}

\section{When $\rho$ is reductive}
Let $M_1$ and $M_2$ be Riemannian manifolds with Riemannian metrics $ds^2_{M_1} = \sum g_{\alpha\beta} dx^\alpha dx^\beta$, 
$ds^2_{M_2} = \sum h_{ij} dy^i dy^j$ respectively.
Let $f\colon M_1\rightarrow M_2$ be a map. 
The energy $E(f)$ of $f$ is defined by
\begin{equation}\nonumber
\frac{1}{2}\int_{M_1} \text{trace}_{ds^2_{M_1}}(f^* ds^2_{M_2}) 
= \frac{1}{2}\int_{M_1}  \sum g^{\alpha\beta}h_{ij}\circ f \frac{\partial f^i}{\partial x^\alpha}\frac{\partial f^j}{\partial x^\beta}.
\end{equation}

The Euler-Lagrange equation for the energy functional $E$ is $\Delta f:=\text{trace} \nabla df =0$, and this can be expressed in a local coordinate by
\begin{equation}\nonumber
\Delta_{M_1}  f^i + \sum \Gamma^i_{jk}\circ f \frac{\partial f^j}{\partial x^\alpha} 
\frac{\partial f^k}{\partial x^\beta}g^{\alpha\beta} =0
\end{equation}
for all $i$, where $\Delta_{M_1} $ is the Laplace-Beltrami operator of ${M_1}$ and $\Gamma^i_{jk}$ is the Christoffel symbol of ${M_2}$.
The map $f$ is said to be {\it harmonic} if $f$ satisfies the Euler-Lagrange equation for the energy functional.

If $M_1$ and $M_2$ are K\"ahler, the map $f \colon M_1\rightarrow M_2$ 
is said to be {\it pluriharmonic} if $\nabla_{1,0}\overline\partial f \equiv 0$.
In local coordinates, pluriharmonic map $f$ satisfies
\begin{equation}\nonumber
\frac{\partial^2 f^i}{\partial \zeta_\alpha  \partial\overline \zeta_\beta} 
+ \sum_{j,k} \Gamma^i_{jk}\circ f 
\frac{\partial f^j}{\partial \zeta^\alpha} 
\frac{\partial f^k}{\partial \overline \zeta^\beta} =0.
\end{equation}
Remark that if $f$ is pluriharmonic, then $f$ is also harmonic and the pluriharmonicity is preserved
by the isometries.

\begin{proof}[Proof of Theorem \ref{main theorem1}]
By Theorem \ref{Corlette}, there exists a $\rho$-equivariant harmonic map from $\widetilde M$ to $\Omega$ and hence there exists a harmonic section of $E$.
By Siu \cite{Siu_1980} and Sampson \cite{Sampson_1986}, 
the given harmonic 
section $s$ is pluriharmonic. For a local coordinate $U$ on $M$ such that 
$\pi^{-1}(U)\cong U\times \Omega$, we may express 
$s|_U(\xi) = (\xi, h(\xi))$ for some $h\colon U\rightarrow \Omega$.
Define a map $\psi$ by
\begin{equation}\nonumber
\psi(\xi,w) = \log\psi_\Omega(h(\xi),w).
\end{equation}

Since $\log \psi_\Omega$ is invariant with respect to the diagonal action of 
$\text{Aut}(\Omega)$ and the automorphisms are isometries with respect to the Bergman metric, 
we may assume that $h(\xi)=0$ and $w$ is contained in a maximal totally geodesic polydisc of $\Omega$.
Consider the pluriharmonic map $(\xi,w)\mapsto\left(h(\xi),w\right)$ from $U\times\Omega$ to 
$\Omega\times\Omega$, with respect to the product metrics $g|_U\otimes g_{sta}$
and $g_\Omega\otimes g_{sta}$ with the standard Euclidean metric $g_{sta}$
of the ambient Euclidean space of $\Omega$, 
the Bergman metric $g_\Omega$ of $\Omega$ and the K\"ahler metric $g$ of $M$. 
Since $\Gamma^k_{ij}(0)=0$ where $\Gamma_{ij}^k$ is 
the Christoffel symbol of $g_\Omega$, pluriharmonicity
of $h$ implies 
\begin{equation}\label{harmonic}
\frac{\partial^2 h^k}{\partial \xi_i\partial \overline \xi_j}(\xi)=0 \quad \text{ for any } i,j,k.
\end{equation}
By the chain rule and the equality
\begin{equation}\nonumber
\begin{aligned}
\frac{\partial^2}{\partial z_k
\partial w_j }\log \psi_\Omega (z,w) 
&= \frac{\partial^2}{ \partial z_k \partial  w_j} \log \frac{K(z,z)K(w,w)}{K(z,w)K(w,z)}=0,
\end{aligned}
\end{equation}
we have 
\begin{equation}\label{zetaw}
\begin{aligned}
\frac{\partial^2}{\partial \xi_i \partial \overline w_j} \log \psi_\Omega(h(\xi), w) 
&=  \sum_k\frac{\partial^2 \log\psi_\Omega}{\partial z_k\partial \overline w_j}(0,w)\frac{\partial h^k}{\partial\xi_i}(\xi) 
+\sum_k \frac{\partial^2 \log\psi_\Omega}{\partial \overline  z_k \partial\overline w_j}(0,w) \frac{\partial \overline h^k}{\partial \xi_i}\\
&=  \sum_k\frac{\partial^2 \log\psi_\Omega}{\partial z_k\partial \overline w_j}(0,w)\frac{\partial h^k}{\partial\xi_i}(\xi),
\end{aligned}
\end{equation}

\begin{equation}\label{ww}
\begin{aligned}
\frac{\partial^2}{\partial w_i \partial \overline w_j} \log \psi_\Omega(h(\xi), w) 
=  \frac{\partial^2 \log\psi_\Omega}{\partial w_i\partial \overline w_j}(0,w),
\end{aligned}
\end{equation}
and by \eqref{harmonic}
\begin{equation}\label{general formula}
\begin{aligned}
\frac{\partial^2}{\partial \xi_i\partial\overline \xi_j}  \log\psi_\Omega(h(\xi),w)
=&\sum_{k,l}\left( \frac{\partial^2 \log\psi_\Omega}{\partial z_k\partial z_l}
\frac{\partial h^k}{\partial \xi_i} 
\frac{\partial h^l}{\partial \overline \xi_j} 
+\frac{\partial^2 \log\psi_\Omega}{\partial \overline z_k\partial z_l}
\frac{\partial \overline h^k}{\partial \xi_i} 
\frac{\partial h^l}{\partial \overline \xi_j} \right.\\
&\quad\quad\quad\left. +\frac{\partial^2 \log\psi_\Omega}{\partial z_k\partial \overline  z_l}
\frac{\partial h^k}{\partial \xi_i} 
\frac{\partial \overline  h^l}{\partial \overline \xi_j} 
+\frac{\partial^2 \log\psi_\Omega}{\partial \overline z_k\partial \overline z_l}
\frac{\partial \overline h^k}{\partial \xi_i} 
\frac{\partial \overline  h^l}{\partial \overline \xi_j}\right).
\end{aligned}
\end{equation}

{\bf Type I, $\Omega_{p,q}$: }
At $(0,w_0)$ with $w_0 = \text{ diag}(w_{11}, \ldots, w_{pp})$, $|w_{jj}|<1$ for all $j$, one obtains
\begin{equation}\nonumber
\begin{aligned}
&\frac{\partial}{\partial z_{ij}} \det(I - z\overline w^t)
= - \overline w_{ij},\\
&\frac{\partial^2}{\partial z_{ij}\partial z_{kl}} \det(I-z\overline w^t) = 
\det \left(\begin{array}{cc}
\overline w_{ij} & \overline w_{kj}\\
\overline w_{il}& \overline w_{kl}
\end{array}\right),\\
\end{aligned}
\end{equation}
and hence
\begin{equation}\label{ijkltypeI}
\begin{aligned}
\frac{\partial^2}{\partial z_{ij}\partial z_{kl} }\log \psi_{\Omega_{p,q}^I} (z,w) 
&= -c_2\frac{\partial^2}{ \partial z_{ij} \partial  z_{kl}} \log 
\frac{\det(I - z\overline z^t) \det (I - w\overline w^t)}{\det(I-z\overline w^t)\det(I- w\overline z^t)}\\
&= c_2\frac{\partial^2}{ \partial z_{ij} \partial  z_{kl}} 
\log \det (I - z\overline w^t)
=-c_2 \overline w_{il} \overline w_{kj}.
\end{aligned}
\end{equation}
Moreover we have
\begin{equation}\label{nobars}
\frac{\partial^2}{\partial z_{ij}\partial\overline w_{kl}} \log \det(I-z\overline w^t) = -\delta_{ik}\delta_{jl}.
\end{equation}
By \eqref{general formula} and \eqref{ijkltypeI}
\begin{equation}\label{zetazeta}
\begin{aligned}
&\frac{1}{c_2}\frac{\partial^2}{\partial \xi_\alpha \partial\overline \xi_\beta}  \log\psi_{\Omega_{p,q}^I}(h(\zeta),w)\\
&= \sum_{i,k=1}^p\left(
- \overline w_{ii}\overline w_{kk}
\frac{\partial h^{ik}}{\partial \xi_\alpha} 
\frac{\partial h^{ki}}{\partial \overline \xi_\beta} 
- 
w_{ii}w_{kk}
\frac{\partial \overline h^{ik}}{\partial \xi_\alpha} 
\frac{\partial \overline  h^{ki}}{\partial \overline \xi_\beta}
\right)
+\sum_{i=1}^p\sum_{j=1}^q\left(
\frac{\partial \overline h^{ij}}{\partial \xi_\alpha} 
\frac{\partial h^{ij}}{\partial \overline \xi_\beta} 
+\frac{\partial h^{ij}}{\partial \xi_\alpha} 
\frac{\partial \overline  h^{ij}}{\partial \overline \xi_\beta}
\right).
\end{aligned}
\end{equation}
As a consequence of \eqref{zetaw},
\eqref{ww}, \eqref{zetazeta} and  \eqref{nobars}
for $Z=\sum_{\alpha} X_{\alpha}\frac{\partial }{\partial \xi_{\alpha}}+ \sum_{i,j} Y_{ij} \frac{\partial}{\partial z_{ij}}$ with 
$X=\sum_{\alpha} X_{\alpha}\frac{\partial }{\partial \xi_{\alpha}}$ we obtain
\begin{equation}\nonumber
\begin{aligned}
&\frac{1}{c_2}\partial\bar\partial \log \psi_{\Omega_{p,q}}(h(\zeta), w)(Z, \overline Z)\\
&=\sum_{i=1}^p\sum_{j=1}^q\left(
 |X \overline h^{ij}|^2
+|X h^{ij}|^2
- 2\text{Re}\left((Xh^{ij})\overline Y_{ij}\right)\right)
-\sum_{i,j=1}^p 2\text{Re}\left( w_{ii}w_{jj}
(X \overline h^{ij})
(\overline X\overline  h^{ji})\right)
+\partial_w\bar\partial_w \log\psi_{\Omega^I_{p,q}}(0,w)(Y,\overline Y)\\
&\geq \sum_{i=1}^p\sum_{j=1}^q\left(
 |X \overline h^{ij}|^2
+|X h^{ij}|^2
- 2\text{Re}\left((Xh^{ij})\overline Y_{ij}\right)
+
\frac{|Y_{ij}|^2}{1-|w_{ii}|^2}\right)
- \sum_{i,j=1}^p2\text{Re}\left( w_{ii}w_{jj}
(X \overline h^{ij})
(\overline X\overline  h^{ji})\right)  \\
&
= \sum_{i,j=1}^p\left(
\left| w_{ii} (X\overline h^{ij} )- \overline w_{jj} (Xh^{ji})\right|^2
+(1-|w_{ii}|^2) |X\overline h^{ij}|^2 \right)
+ \sum_{i=1}^p\sum_{j=p+1}^q 
\left( |X\overline h^{ij}|^2 + |w_{ii}|^2 |Xh^{ij}|^2\right)\\
&\quad
+ \sum_{i=1}^p\sum_{j=1}^q\left|
\sqrt{1-|w_{ii}|^2 } (Xh^{ij})
- \frac{Y_{ij}}{\sqrt{1-|w_{ii}|^2}}\right|^2
\geq 0
\end{aligned}
\end{equation}
by \eqref{plurisubharmonicity}.
Hence $\partial\bar\partial \log \psi_{\Omega_{p,q}^I}(h(\xi), w)$ is 
positive semi-definite.

{\bf Type II, III:}
Since $\psi_{\Omega^{II}_n}$ and $\psi_{\Omega^{III}_n}$ are the restriction of 
$\psi_{\Omega^I_{n,n}}$ and $\frac{1}{2}\psi_{\Omega^I_{n,n}}$ on $\Omega^{II}_n$ and $\Omega^{III}_n$ respectively,
$\partial\bar\partial \log \psi_{\Omega_{n}^{II}}(h(\xi), w)$, $\partial\bar\partial \log \psi_{\Omega_{n}^{III}}(h(\xi), w)$ are 
positive semi-definite.

{\bf Type IV:}
At $(0,w)$, we have
\begin{equation}\nonumber
\begin{aligned}
&\frac{\partial^2}{\partial z_i \partial z_j} \log \psi_{\Omega_{n}^{IV}}(z,w)
=c_2\frac{\partial^2}{\partial z_i \partial z_j} \log ( 1-2zw^* + (zz^t)(\overline{ww^t}))
=2\delta_{ij} (\overline{ww^t})-4\overline w_i \overline w_j. 
\end{aligned}
\end{equation}
Therefore we have 
{\small
\begin{equation}\nonumber
\frac{\partial^2}{\partial z_1 \partial z_1} \log \psi_{\Omega_{n}^{IV}}
= -2c_2(\overline w_1^2 - \overline w_2^2),\quad
\frac{\partial^2}{\partial z_2 \partial z_2} \log \psi_{\Omega_{n}^{IV}}
= 2c_2(\overline w_1^2 - \overline w_2^2), \quad
\frac{\partial^2}{\partial z_1 \partial z_2} \log \psi_{\Omega_{n}^{IV}}
=-4c_2\overline w_1\overline w_2
\end{equation}}
and hence by \eqref{general formula}
\begin{equation}\nonumber
\begin{aligned}
&\frac{1}{2c_2}\frac{\partial^2}{\partial \xi_\alpha \partial\overline \xi_\beta}  \log\psi_{\Omega_{n}^{IV}}(h(\xi),w)
=
\sum_j \left(
\frac{\partial \overline h^{j}}{\partial \xi_\alpha} 
\frac{\partial h^{j}}{\partial \overline \xi_\beta} 
+
\frac{\partial h^{j}}{\partial \xi_\alpha} 
\frac{\partial \overline  h^{j}}{\partial \overline \xi_\beta} 
\right)+ \\
& \quad\quad2\text{Re}\left( -(\overline w_1^2-\overline w_2^2)\frac{\partial h^{1}}{\partial \xi_\alpha} 
\frac{\partial h^{1}}{\partial \overline \xi_\beta} 
+(\overline w_1^2-\overline w_2^2)\frac{\partial  h^{2}}{\partial \xi_\alpha} 
\frac{\partial   h^{2}}{\partial \overline \xi_\beta}
-2\overline w_1\overline w_2 
\frac{\partial  h^{1}}{\partial \xi_\alpha} 
\frac{\partial   h^{2}}{\partial \overline \xi_\beta} 
-2\overline w_1\overline w_2 
\frac{\partial  h^{2}}{\partial \xi_\alpha} 
\frac{\partial   h^{1}}{\partial \overline \xi_\beta} \right)
\end{aligned}
\end{equation}
This implies for $X = \sum_{\alpha=1}^n X_\alpha \frac{\partial }{\partial \xi_\alpha}$
\begin{equation}\nonumber
\begin{aligned}
&\frac{1}{2c_2}\sum_{\alpha,\beta}
X_\alpha \frac{\partial^2}{\partial \xi_\alpha \partial\overline \xi_\beta}  \log\psi_{\Omega_{n}^{IV}}(h(\xi),w)\overline X_\beta\\
&= \sum_{j=1}^n \left( |X \overline h^{j}|^2
+|X h^{i}|^2\right)
+ 2\text{Re} \left(
-(\overline w_1^2 - \overline w_2^2) (Xh^1)(\overline Xh^1)
+(\overline w_1^2 - \overline w_2^2) (Xh^2) (\overline Xh^2)\right.\\
&\quad \quad \quad \quad \quad \quad \quad \quad \quad \quad \quad \quad \quad \quad \quad \quad \quad 
\left. -2\overline w_1 \overline w_2 (Xh^1)(\overline X h^2) -2\overline w_1\overline w_2 (Xh^2) (\overline X h^1)\right).
\end{aligned}
\end{equation}
For $\zeta_1$, $\zeta_2$ satisfying $w_1 = \lambda(\zeta_1 + \zeta_2)$ and $w_2 = i\lambda(\zeta_1 - \zeta_2)$ and for 
$X=\sum X_\alpha\frac{\partial}{\partial \xi_\alpha}$, we have 
\begin{equation}\label{zetazeta4}
\begin{aligned}
&\frac{1}{2c_2}\partial_\xi \bar\partial_\xi \log\psi_{\Omega_{n}^{IV}}(h(\xi), w) (X,\overline X) \\
&=  \sum_{j=1}^n \left( |X \overline h^{j}|^2
+|X h^{j}|^2\right)
+ 2\text{Re} \left(
\frac{i}{2}(\overline \zeta_1^2 + \overline \zeta_2^2) (Xh^1)(\overline Xh^1)
-\frac{i}{2}(\overline \zeta_1^2 + \overline \zeta_2^2) (Xh^2) (\overline Xh^2)\right.\\
&\left. 
\quad\quad\quad\quad\quad\quad
\quad\quad\quad\quad\quad\quad\quad\quad\quad\quad\quad\quad
+ \frac{1}{2}(\overline \zeta_1^2 - \overline \zeta_2^2) (Xh^1)(\overline X h^2) 
+\frac{1}{2}(\overline \zeta_1^2 - \overline \zeta_2^2)(Xh^2) (\overline X h^1)\right)\\
&=   \sum_{j=3}^n \left( |X \overline h^{j}|^2
+|X h^{j}|^2\right)
+ \frac{1}{2}\left(
|Xh^1 - iXh^2|^2 
+ |\overline Xh^1 - i\overline X h^2 |^2
+ | Xh^1 + iXh^2|^2
+ |\overline Xh^1 + i \overline Xh^2|^2 
\right)
\\
&\quad
+ 2\text{Re} \left(
\frac{\overline \zeta_1^2}{2}
i(Xh^1 - i Xh^2)(\overline X h^1 - i\overline X h^2)
+ \frac{\overline \zeta_2^2}{2}
i (Xh^1 + iXh^2)(\overline Xh^1 + i \overline Xh^2)
\right)\\
&=  \frac{1}{2}\left[ (1-|\zeta_1|^2)
\left(
|Xh^1 - iXh^2|^2 
+ |\overline Xh^1 - i\overline X h^2 |^2\right)
+(1-|\zeta|^2)\left( | Xh^1 + iXh^2|^2
+ |\overline Xh^1 + i \overline Xh^2|^2 \right)
\right]
\\
&\quad+ \frac{|\zeta_1|^2}{2}\left|
(Xh^1 - iXh^2) + ( X\overline h^1 +i X\overline h^2)\right|^2
+ \frac{|\zeta_2|^2}{2}\left|
(Xh^1 + iXh^2) + ( X\overline h^1 -i X\overline h^2)\right|^2\\
&\quad +\sum_{j=3}^n \left( |X \overline h^{j}|^2
+|X h^{j}|^2\right).
\end{aligned}
\end{equation}
As a consequence of \eqref{zetaw}, 
\eqref{ww}, \eqref{zetazeta4}, \eqref{ddbar type4}, for 
$Z=\sum X_\alpha\frac{\partial}{\partial \xi_\alpha} + \sum Y_{k}\frac{\partial}{\partial w_{k}}$, one obtains
\begin{equation}\label{psh type4}
\begin{aligned}
&\frac{1}{2c_2}\partial\bar\partial \log \psi_{\Omega_n^{IV}}(h(\xi), w)(Z,Z)\\
&=  \frac{1}{2}\left[ (1-|\zeta_1|^2)
\left(
|Xh^1 - iXh^2|^2 
+ |\overline Xh^1 - i\overline X h^2 |^2\right)
+(1-|\zeta_2|^2)\left( | Xh^1 + iXh^2|^2
+ |\overline Xh^1 + i \overline Xh^2|^2 \right)
\right]\\
&\quad
+ \frac{|\zeta_1|^2}{2}\left|
(Xh^1 - iXh^2) + ( X\overline h^1 +i X\overline h^2)\right|^2
+ \frac{|\zeta_2|^2}{2}\left|
(Xh^1 + iXh^2) + ( X\overline h^1 -i X\overline h^2)\right|^2
\\
&\quad
-2\sum_{j=1}^n \text{Re} ( X  h^j) \overline Y_j + \frac{M_{11}}{2}|Y_1|^2 + \frac{M_{22}}{2}|Y_2|^2 + \frac{M_{12}}{2} Y_1\overline Y_2 + \frac{M_{21}}{2}Y_2\overline Y_1 \\
&\quad
+ \frac{\sum_{j=3}^n |Y_j|^2}{(1-|\zeta_1|^2)(1-|\zeta_2|^2)}
+\sum_{j=3}^n \left( |X \overline h^{j}|^2
+|X h^{j}|^2\right)
\end{aligned}
\end{equation}
\begin{equation}
\begin{aligned}
&=
\frac{1}{2}\left[ (1-|\zeta_1|^2)
\left(
|Xh^1 - iXh^2|^2 
+ |\overline Xh^1 - i\overline X h^2 |^2\right)
+(1-|\zeta_2|^2)\left( | Xh^1 + iXh^2|^2
+ |\overline Xh^1 + i \overline Xh^2|^2 \right)
\right]\\
&\quad
+ \frac{|\zeta_1|^2}{2}\left|
(Xh^1 - iXh^2) + ( X\overline h^1 +i X\overline h^2)\right|^2
+ \frac{|\zeta_2|^2}{2}\left|
(Xh^1 + iXh^2) + ( X\overline h^1 -i X\overline h^2)\right|^2
\\
&\quad
-2\sum_{j=1}^2 \text{Re}(Xh^j) \overline Y_j + \frac{M_{11}}{2}|Y_1|^2 + \frac{M_{22}}{2}|Y_2|^2 + \frac{M_{12}}{2} Y_1\overline Y_2 + \frac{M_{21}}{2}Y_2\overline Y_1 \\
&\quad
+ \sum_{j=3}^n \left| \frac{Y_j}{\sqrt{(1-|\zeta_1|^2)(1-|\zeta_2|^2)}}- \sqrt{(1-|\zeta_1|^2)(1-|\zeta_2|^2)}X_jh^j\right|^2\\
&\quad
+ \sum_{j=3}^n\left( |X\overline h^j|^2 + |Xh^j|^2 -{(1-|\zeta_1|^2)(1-|\zeta_2|^2)}|Xh^j|^2\right)
\end{aligned}
\end{equation}
where $M_{11}=M_{22}=\frac{2(1-|\zeta_1|^2)(1-|\zeta_2|^2) + (|\zeta_1|^2 - |\zeta_2|^2)^2}{(1-|\zeta_1|^2)^2(1-|\zeta_2|^2)^2}$
and $M_{12}=-M_{21}=i\frac{(2-|\zeta_1|^2-|\zeta_2|^2)  (|\zeta_1|^2 - |\zeta_2|^2)}{(1-|\zeta_1|^2)^2(1-|\zeta_2|^2)^2}$.
Since we have 
$$|Xh^1|^2 + |Xh^2|^2
= \frac{1}{2} \left( |Xh^1 - i Xh^2|^2 + |Xh^1 + iXh^2|^2\right),$$
and
\begin{equation}\nonumber
\begin{aligned}
&-2\sum_{j=1}^2 \text{Re}(Xh^j) \overline Y_j + \frac{M_{11}}{2}|Y_1|^2 + \frac{M_{22}}{2}|Y_2|^2 + \frac{M_{12}}{2} Y_1\overline Y_2 + \frac{M_{21}}{2}Y_2\overline Y_1 \\
&\geq \sum_{j=1}^2 \left( -2\text{Re}(Xh^j) \overline Y_j + \frac{|Y_j|^2}{S^{IV}_n}\right)
= \sum_{j=1}^2 \left(\left| \sqrt{S^{IV}_n} Xh^j - \frac{Y_j}{\sqrt{S^{IV}_n}}\right|^2
-S^{IV}_n|Xh^j|^2\right),
\end{aligned}
\end{equation}
\eqref{psh type4} is greater than or equal to $0$.

{\bf Type V, VI:} We omit the proof.
\smallskip

Hence  $\psi$ is psh.
Since $\log\psi_\Omega$ is invariant under the diagonal action of
$\text{Aut}(\Omega)$, $\psi$ is well defined on $M$.
By the construction of $\psi_\Omega$, it is an exhaustion function.
\end{proof}

\section{When $\rho$ is non-reductive}

Consider the heat equation according to Eells-Sampson given in \cite{Eells_Sampson_1964}:
\begin{equation}\label{heat equation}
\begin{aligned}
\frac{d}{dt}s(p,t) &=\Delta s(p,t),\\
s(p,0) &= s_0(p),
\end{aligned}
\end{equation}
for a map $s\colon M \times [0,\tau)\rightarrow E$ with $\tau>0$
where $s_0\colon   M\rightarrow E$ is a continuous section of $E$.
Note that since $\Omega$ is contractible, there exists such $s_0$ (see \cite{Steenrod} for example).
Denote $g_\Omega$ the Bergman metric on $\Omega$, $g_M$ the K\"ahler metric on $M$, and $g_E$ the induced metric from $g_\Omega$, $g_M$ on $E$. 
Let $d_\Omega$ and $d_E$ denote the distances induced from $g_\Omega$ and $g_E$ respectively.

\begin{lemma}[Hamilton \cite{Hamilton}, Diederich-Ohsawa \cite{Diederich_Ohsawa_1985}]
The family $\{s_t = s(\cdot, t)\}$ is well defined for any $t\in \mathbb R^+$ and $s_t$ 
is also a section for any $t$. 
Moreover the family $\{s_t = s(\cdot, t): t\in \mathbb R^+\}$ is uniformly 
equicontinuous on $ M$ 
with respect to $g_E$ and $g_M$.
\end{lemma}
Let $\widehat \Omega$ denote the compact dual of $\Omega$.
We can choose $0<t_1<t_2<\cdots $ with 
$\lim_{k\rightarrow\infty} t_k = \infty$
so that  
$$s_\infty\colon   M\rightarrow  M\times_\rho \widehat \Omega
,\quad\quad s_\infty(p) := \lim_{k\rightarrow \infty} s_{t_k}(p)
$$
is harmonic (cf. \cite{Diederich_Ohsawa_1985, Hamilton}).

\begin{lemma}\label{converge to a boundary component}
Let $\{p_j\}_{j=1}^\infty$, $\{q_j\}_{j=1}^\infty$ be sequences in $\Omega$ and 
$p$, $q$ be points on $\partial \Omega$ such that $p_j\rightarrow p$, 
$q_j\rightarrow q$ as $j\rightarrow \infty$.
If $\liminf_{j\rightarrow \infty} d_\Omega(p_j, q_j) <\infty $, then 
$p$ and $q$ belong to the same boundary component.
\end{lemma}
\begin{proof}
Since the Bergman distance and the Kobayashi distance are equivalent on $\Omega$,
the condition $\liminf_{j\rightarrow \infty} d_\Omega(p_j, q_j) <\infty $ implies 
$\liminf_{j\rightarrow \infty} d^K_\Omega (p_j, q_j) <\infty $ with the Kobayashi distance $d^K_\Omega$ on $\Omega$.
By Proposition 3.5 in \cite{Zimmer_2017}, for a complex line $L$ containing $p$ and $q$,
the interior of $\overline \Omega \cap L$ in $L$ contains $p$ and $q$.
Since the interior of $\overline \Omega\cap L$ should be contained in the boundary
component of $\Omega$, we obtain the lemma.
\end{proof}

\begin{proposition}\label{contained in maximal parabolic}
If $\rho$ is non-reductive, then ${\rho(\pi_1(M))}$ is contained in a maximal real parabolic subgroup in  $\text{Aut}(\Omega)$.
\end{proposition}

\begin{proof}
Since $\rho$ is non-reductive,
by Theorem \ref{Corlette} 
there exists no harmonic section from $M$ to $ M\times_\rho  \Omega$.
This implies that there exists
a family $\{s_t = s(\cdot, t) : t\in [0,\infty)\}$ satisfying \eqref{heat equation}
which is uniformly equicontinuous on $  M$ 
with respect to $g_{M}$ and $g_E$.

Let $p,\,q\in M$. Choose $t_1<t_2<\cdots$
with $\lim_{k\rightarrow \infty} t_k=\infty$ so that $s_\infty:= \lim_{k\rightarrow \infty}
s_{t_k}$ is harmonic. For simplicity we will denote $s_{t_k}$ by $s_k$.
Then we have $\lim_{k\rightarrow\infty}d_E(s_{k}(p), s_{k}(q))<\infty$ and 
$\lim_{k\rightarrow \infty} s_{k}(p)$, $\lim_{k\rightarrow \infty} s_{k}(q) \in  M\times_\rho \partial \Omega$.
Denote $$
s_\infty(p) = [\tilde p, z], \quad s_\infty(q)=[\tilde q, w], \quad 
s_t(p) = [\tilde p, z_t]\quad \text{ and }\quad  s_t(q) = [\tilde q, w_t], $$
where $z,w\in \partial \Omega, z_t, w_t\in \Omega$ and 
$\tilde p, \tilde q \in \widetilde M$.
Since we have 
\begin{equation}\nonumber
\begin{aligned}
\infty> \lim_{t\rightarrow\infty}d_E \left( s_t(p), s_t(q) \right) 
&= \lim_{t\rightarrow \infty}d_E \left( [\tilde p, z_t], [\tilde q, w_t] \right)\\
&\geq \lim_{t\rightarrow \infty} \min_{\gamma\in \pi_1(M)} 
\left( d_M \left(\gamma \tilde p, \gamma \tilde q \right)+d_\Omega \left(\rho(\gamma) z_t, \rho(\gamma) w_t \right) \right),
\end{aligned}
\end{equation}
there exists $\gamma\in \pi_1(M)$ such that 
$d_\Omega(\rho(\gamma) z, \rho(\gamma) w)<\infty$.
Hence $z$ and $w$ should be contained in a boundary component of $\partial \Omega$, say $B$, 
by Lemma \ref{converge to a boundary component}.

Let $\gamma\in \pi_1(M)$. Since 
we have $s_\infty(p) = [\tilde p, z] = [\gamma\tilde p, \rho(\gamma) z]$ for any $p\in M$,
$z$ and $\rho(\gamma) z$ belongs to $B$.
Since automorphisms of $\Omega$ permute boundary components
and $z, \,\rho(\gamma)z$ both belong to $B$, 
$\rho(\gamma)$ is a normalizer
of $B$ which is a maximal real parabolic subgroup of  $\text{Aut}(\Omega)$.
\end{proof}
\begin{remark}
Let $\Omega=\Omega_1\times\cdots \times \Omega_k$
be a bounded symmetric domain with irreducible factors 
$\Omega_i$, $i=1,\ldots, k$.
If $\rho$ is non-reductive,
then ${\rho(\pi_1(M))}$ is contained in $P_1\times \cdots\times P_k$
where $P_i$ is a maximal parabolic subgroup in  $\text{Aut}(\Omega_i)$ or  $\text{Aut}(\Omega_i)$ itself for each $i$.
\end{remark}
\begin{lemma}\label{converge to}
Suppose that $\rho$ is non-reductive.
Then a limit map $s_\infty$ 
of the family $\{s_t=s(\cdot, t) : t\in [0,\infty)\}$ 
is a harmonic map with respect to the induced metric from the K\"ahler metric on $M$ and the Bergman metric 
of the boundary component where $s_t$ converges to.
\end{lemma}
\begin{proof}
Since the Bergman metric on the boundary component is the limit of $g_\Omega$ restricted 
to the corresponding characteristic symmetric subspaces of it, we obtain the lemma.
\end{proof}

\begin{lemma}\label{product}
For $\Lambda\subset \Pi$ with $|\Lambda|=1$, consider a totally geodesic subspace $\Delta\times X_{\Pi-\Lambda,0}\subset \Omega$.
Let $\sigma\colon \Omega\to\frak m_{\Lambda,1}^-$ be the projection.
If $\sigma(p)$ tends to $\partial \Delta$, then $p$ tends to $\partial \Delta\times \overline{X_{\Pi-\Lambda,0 }}$.
\end{lemma}
\begin{proof}
Let $\frak m^- = \mathbb C\alpha \oplus \mathcal H_\alpha \oplus\mathcal N_\alpha$ be the decomposition with respect to the tangent unit vector $\alpha\in T_0 \Delta$
such that the bisectional curvature of $\Omega$ with respect to the Bergman metric in directions $\alpha$ and $\xi \in \mathcal H_\alpha$ (resp. $\xi\in \mathcal N_\alpha$) equals to $1/2$ (resp. $0$) where $\xi$ is a root vector.
Remark that $\mathbb C \alpha \cap \Omega\cong \Delta$ and $\mathcal N_\alpha\cap \Omega\cong X_{\Pi-\Lambda,0}$. (For more detail, see \cite{Mok_1989}.) Therefore we only need to show that $\mathcal H_\alpha$-component vanishes as the modulus of $\mathbb C\alpha$-component tends to $1$.
For each root vector of unit norm $\xi\in \mathcal H_\alpha$, either 
\begin{enumerate}
\item $(\mathbb C\alpha +\mathbb C\xi)\cap \Omega\cong \mathbb B^2$ is totally geodesic in $\Omega$, or 
\item
there exists $\eta\in \mathcal N_\alpha$ such that $(\mathbb C\alpha + \mathbb C \xi + \mathbb C \eta)\cap \Omega \cong \Omega_3^{IV}$ is totally geodesic in $\Omega$
\end{enumerate}
 by Proposition 3.6 in \cite{Mok_Tsai_1992}. If $(\mathbb C\alpha +\mathbb C\xi)\cap \Omega\cong \mathbb B^2$, then $\mathbb C\xi$-component of $p$ tends to zero since the modulus of $\mathbb C\alpha$-component tends to $1$. 
Now let us consider the case (2). By \eqref{note4} for each $w=(w_1, w_2,w_3)$ with $|\zeta_1|=1$ we have 
\begin{equation}\nonumber
\begin{aligned}
0\leq 1-2w\overline w^t + |ww^t|^2 &= (1-|\zeta_1|^2)(1-|\zeta_2|^2) - 2|w_3|^2 + |i\zeta_1\zeta_2 + w_3^2|^2 - | i\zeta_1\zeta_2|^2\\
&=- 2|w_3|^2 + |w_3|^4 + 2\text{Re}(i\zeta_1\zeta_2 \overline w_3^2)
\leq - 2|w_3|^2 + |w_3|^4 + 2|\zeta_2\overline w_3^2|,
\end{aligned}
\end{equation}
which implies that \begin{equation}\label{def2}
2\leq |w_3|^2 + 2|\zeta_2|^2,
\end{equation}
and 
\begin{equation}\label{def1}
1\geq w\overline w^t = \frac{1}{2} + \frac{1}{2} |\zeta_2|^2 + |w_3|^2
\end{equation}
since $w_1^2 + w_2^2 = i\zeta_1\zeta_2$ for $w_1 = \lambda(\zeta_1 + \zeta_2)$, $w_2 = i\lambda(\zeta_1 - \zeta_2)$ with $\lambda^2 = i/4$.
By \eqref{def1} and \eqref{def2}, we have 
\begin{equation}\nonumber
2\leq |w_3|^2 + 2|\zeta_2|^2
\leq |w_3|^2 + 2(1-2|w_3|^2)
\end{equation}
and we induce $w_3=0$. This completes the proof.
\end{proof}

For a non-reductive representation $\rho\colon \pi_1(M)\to \text{Aut}(\Omega)$, let $B$ be the boundary component to where a subsequence of a family of solutions of the heat equation \eqref{heat equation} converges. Denote by $N(B):=\{g\in \text{Aut}(\Omega) : gB=B\}$ the set of normalizers of $B$ in  $\text{Aut}(\Omega)$ and by
$c_B$ the Cayley transformation with respect to $B$.
Define $\rho_c \colon \pi_1(M)\rightarrow Ad_{c_B} N(B)\subset G^\mathbb C$ by \begin{equation}\nonumber
\rho_c(\gamma) :=c_B\circ \gamma\circ c_B^{-1} \colon c_B(\Omega)\rightarrow c_B(\Omega).
\end{equation}
Then $ M\times_{\rho_c} c_B(\Omega)$ is holomorphically equivalent 
to $ M \times_\rho \Omega $ as a fiber bundle over $M$
and $c_B(\Omega)$ has the form \eqref{c(Omega)} in Theorem \ref{Cayley transform}.

\begin{theorem}\label{main theorem2}
Let $E = M \times_\rho \Omega$ be a holomorphic fiber bundle over a compact K\"ahler manifold $M$
with an irreducible bounded symmetric domain fiber $\Omega$ where $\rho\colon \pi_1(M)\rightarrow \text{Aut}(\Omega)$ is a non-reductive representation.
Suppose that there exists a family of solutions of \eqref{heat equation}
which has a subsequence that converges to a 1-component.
Then there exist a plurisubharmonic function $\psi$ on $E$ and a
totally geodesic subspace $\Delta\times\Delta^\perp\subset \Omega$ which is invariant under $\rho(\pi_1(M))$ such that 
$$\lim_{[p,z]\rightarrow \partial E \setminus  M\times_\rho (\partial\Delta \times \Delta^\perp)}\psi([p,z])= \infty$$
where $\{e^{i\theta}\}\times \Delta^\perp$ is a maximal boundary component of $\Omega$ for each $\theta\in \mathbb R$.
\end{theorem}

\begin{proof}[Proof of Theorem \ref{main theorem2}]
By the assumption and Lemma \ref{converge to} there exist a maximal boundary component $B$ and a $\rho$-equivariant 
harmonic map from $\widetilde M$ to $B$ with respect to the 
K\"ahler metric induced from $M$ and the Bergman metric on $B$.
Moreover $\rho(\pi_1(M))\subset N(B)$
and any element in $\rho(\pi_1(M))$ has the decomposition of the form \eqref{decomposition}.
Hence there exists a psh exhaustion function, 
say $\psi_B$, on $M\times_\rho B\cong M\times_{\rho_c} X_{\Lambda, 0}$ by Theorem \ref{main theorem1}.

From what follows we will denote $K_\Omega(z,z)$ by $K_\Omega(z)$ for a domain $\Omega$ for simplicity.

Since $B$ is a 1-component, we have $|\Lambda|=1$ and hence $\dim \frak m_{\Pi - \Lambda, 1}^- = 1$.
Let $\mathbb H$ denote the upper half space $\{ z\in \mathbb C : \text{Im } z >0\}$.
Define a function $\delta\colon M\times _{\rho_c} c_B(\Omega)\rightarrow \mathbb R$ by
\begin{equation}\label{delta}
\delta([p, (E_1, E_2, E_3)]):=
\frac{1}{2a}\log \left( \frac{K_{c_B(\Omega)}(E_1,E_2,E_3)}{K_{\mathbb H}(E_1)^a K_{X_{\Lambda,0}}(E_3)}\right)
\end{equation}
where $a = 1 + \frac{1}{2}\dim \frak m^{\Lambda,-}_2$.
First we claim that $\delta$ is a well defined function. 
For a biholomorphism $f\colon D_1\rightarrow D_2$ with domains $D_1$, $D_2$ in $\mathbb C^n$,
we have $K_{D_1}(z,z)= |\det J(f)(z)|^2 K_{D_2}(f(z), f(z)),$
where we denote by $J(f)$ the Jacobian matrix of $f$.
Hence, for the claim to hold, we only need to show
\begin{equation}\label{det}
\left|\det( J(g))\right| = \left|\det(  J(g|_{\frak m_\Lambda^-})) \det ( J  (g|_{\frak m_{\Pi - \Lambda, 1}^- } ) )^a\right|
\end{equation}
for any $g\in \rho_c(\pi_1(M))$.
Since $\rho_c(\pi_1(M))\subset \text{Ad}_{c_B}N(B)$, we may decompose $g$ by $g_1g_2g_3$ with $g_1\in G_\Lambda$, $g_2\in L_2^\Lambda K^*_{\Pi-\Lambda,1}$, $g_3 \in N^{\Lambda,-}$.

Since $g_3$ acts on $c_B(\Omega)$ by the expression \eqref{translate}, the left and the right hand sides of the equation \eqref{det} are both equal to $1$ and hence \eqref{det} holds for $g_3$.

Note that $L_2^{\Lambda} K^*_{\Pi -\Lambda,1}$ acts on $\frak m^-$ by the 
adjoint representation and it preserves 
$\frak m_{\Pi -\Lambda, 1}^-$, $\frak m_2^{\Lambda, -}$ and $\frak m_\Lambda^-$. 
Since the action of $g_2$ on $\frak m_{\Pi -\Lambda, 1}^-$ is real, there exists a positive constant $C$ such that 
$C\text{Im }(E_1) = \text{Im }(g_2 E_1)$ for any $E_1$. 
This implies that the right hand side of
the equation \eqref{det} equals to $C^a$. 
On the other hand since $g_2$ acts on $\frak m_{\Lambda}^-$ trivially, by \eqref{bilinear form} we have 
\begin{equation}\label{g2}
C\left( \text{Im}\, E_1 - \text{Re}\, F_{E_3}(E_2, E_2) \right)
= \text{Im}\, (g_2 E_1) - \text{Re}\, F_{E_3}(g_2 E_2, g_2 E_2),
\end{equation}
which implies that $C \text{ Re}\, F_{E_3}(E_2, E_2)  = \text{Re}\, F_{E_3}(g_2 E_2, g_2 E_2)$. Hence the action of $g_2$ on $\frak m_2^{\Lambda, -}$ is equivalent to the action of $\sqrt{ C} U$ for some unitary transformation $U$ on $\frak m_2^{\Lambda, -}$.  
As a result $|\det ( J(g))|$ equals to $C^a$.

Remark that $G_\Lambda$ acts on $\frak m_{\Pi-\Lambda,1}^-$ trivially.
Let $g_1=\text{exp}(m_\Lambda^+)k\,\text{exp}(m_\Lambda^-)\in G_\Lambda$
where $m_\Lambda^+ \in \frak m_\Lambda^+$, $m_\Lambda^- \in \frak m_\Lambda^-$
and $k\in K^\mathbb C$. 
Since $\text{exp}(m_\Lambda^-)$ acts on $\frak m^-$ as a translation and $k$ belongs to the isotropy subgroup at $0$, we have $\det(J(k\,\text{exp}(m_\Lambda^-)) = 1$.
Hence we only need to consider the case when $g=\text{exp}( m^+_\Lambda$).
Let $E=(E_1,E_2,E_3)\in \frak m^-= \frak m_{\Pi-\Lambda,1}^- +  \frak m_2^{\Lambda,-}
+\frak m_\Lambda^-$ and 
$e=(e_1,e_2,e_3)\in T_E \frak m^- \cong \frak m_{\Pi-\Lambda,1}^- 
+  \frak m_2^{\Lambda,-} +\frak m_\Lambda^-$.
By the Hausdorff-Campbell formula, we have  
\begin{equation}\nonumber
\begin{aligned}
&\text{exp}(m^+_\Lambda) \text{exp}(E+te)\cdot K^\mathbb C M^+\\
&= \text{exp}\left( E+ te + [m^+_\Lambda, E+te] + \frac{1}{2} [m^+_\Lambda,[m^+_\Lambda, E+te]]
+\cdots\right) K^\mathbb C M^+ .\\
\end{aligned}
\end{equation}
By Lemma \ref{bracket} we obtain
\begin{equation}\nonumber
\begin{aligned}
&E+ te + [m^+_\Lambda, E+te] + \frac{1}{2} [m^+_\Lambda,[m^+_\Lambda, E+te]]+\cdots\\
&= \text{Ad}_{\text{exp}(m^+_\Lambda)}(E) + t\left(
e_1 + e_2 + [m^+_\Lambda, e_2] + \text{Ad}_{\text{exp}(m_\Lambda^+)}(e_3) \right),
\end{aligned}
\end{equation}
and \begin{equation}\nonumber
[m^+_\Lambda, e_2]\in \frak k^\mathbb C.
\end{equation}
As a consequence, we obtain \eqref{det} for $g_1$ and hence we complete the proof of the claim.

On the other hand, the relation \eqref{det} implies that for any $g_3\in G_{\Lambda}$
we have 
$$
K_{c_B(\Omega)}(i,0,E_3) = K_{c_B(\Omega)}(i,0,g_3E_3)\left| \det J (g_3|_{\frak m_{\Lambda}^-})(E_3)\right|^2
$$
which implies that there exists a positive constant $\nu$ such that
\begin{equation}\nonumber
K_{c_B(\Omega)}(i, 0, E_3)
=\nu K_{X_{\Lambda,0}}(E_3)
\end{equation}
for any $E_3\in X_{\Lambda,0}$ since $K_{X_{\Lambda,0}}(E_3) = K_{X_{\Lambda,0}}(g_3E_3)| \det J (g_3|_{\frak m_{\Lambda}^-})(E_3)|^2$.
As a consequence we have 
\begin{equation}\nonumber
\begin{aligned}
K_{c_B(\Omega)}(E_1, E_2, E_3 ) 
&= K_{c_B(\Omega)}(i(\text{Im }E_1 - \text{Re }F_{E_3}(E_2, E_2)), 0, E_3)\\
&= \frac{K_{c_B(\Omega)}(i, 0, E_3)}{\left( \text{Im }E_1 - \text{Re }F_{E_3}(E_2, E_2)\right)^{2a}}\\
&= \nu K_{X_{\Lambda,0}}(E_3)K_{\mathbb H}\left(i( \text{Im }E_1 - \text{Re }F_{E_3}(E_2, E_2))\right)^a.
\end{aligned}
\end{equation}
Since 
\begin{equation}\label{delta}
\begin{aligned}
\delta([p, (E_1, E_2, E_3)])&=
\frac{1}{2a}\log \left( \frac{K_{c_B(\Omega)}(E_1,E_2,E_3)}{K_{\mathbb H}(E_1)^a K_{X_{\Lambda,0}}(E_3)}\right)\\
&= \frac{1}{2a} \log \left( \frac{\nu K_{\mathbb H}(i(\text{Im }E_1 - \text{Re } F_{E_3}(E_2, E_2)))^a }{K_{\mathbb H}(E_1)^a }\right) \\
&= 
-\log \left( \frac{\text{Im }E_1 - \text{Re } F_{E_3}(E_2, E_2) }{\text{Im }E_1}\right)
+ \frac{1}{2a}\log \nu,
\end{aligned}
\end{equation}
one obtains that $\delta([p, (E_1, E_2, E_3)])$ diverges to infinity whenever $\text{Im }E_1 - \text{Re } F_{E_3}(E_2, E_2)$ tends to zero unless $E_2= 0$.
Remark that $\text{Re }F_{E_3}$ is a positive definite bilinear form since $c_B(\Omega)$ is Kobayashi hyperbolic.
Hence by a straightforward calculation, the last expression in \eqref{delta} is psh.

Now define a function on $ M\times_{\rho_c} c_B(\Omega)$ by
\begin{equation}\nonumber
\begin{aligned}
\psi_{c_B(\Omega)} ([p, (E_1, E_2, E_3)]) := &\psi_B([p, E_3]) 
+\delta([p, (E_1, E_2, E_3)]).
\end{aligned}
\end{equation}
Let $\Delta_\Lambda$ be the unit disc such that $\Delta_\Lambda\times X_{\Lambda,0}$ is a totally geodesic subspace in $\Omega$
and let $\sigma\colon\Omega\rightarrow \Delta_\Lambda$ be a projection to $\frak m_{\Pi-\Lambda,1}^-$. Note that $\Delta_\Lambda\subset \frak m_{\Pi-\Lambda,1}^-$ and $\sigma(\Omega)=\Delta_{\Lambda}$.
For $(E_1, E_2,E_3)\in c_B(\Omega)$, if $|E_1|$ tends to infinity or $\text{Im }E_1$ tends to zero, then 
\begin{equation}\nonumber
\sigma\circ c_B^{-1}(E_1,E_2,E_3) \to \partial \Delta_\Lambda,
\end{equation}
since $c_B$ maps $\Delta_\Lambda$ onto $\mathbb H=\{(E_1,0,0)\in \mathfrak m^- : \text{Im }E_1 >0\}$ biholomorphically.
Hence $c_B^{-1}(E_1,E_2,E_3)$ converges to $\partial \Delta_\Lambda\times\overline{X_{\Lambda,0}}$ by Lemma \ref{product}
and hence the induced function $\psi$ from $\psi_{c_B(\Omega)}$ to $M\times_\rho \Omega$ can be finite only when $[p,z]$ tends to $M\times_\rho(\partial \Delta_\Lambda\times X_{\Lambda,0})$ since $\psi_B([p,E_3])$ diverges to infinity as $E_3$ converges to $\partial X_{\Lambda, 0}$. As a result it is a desired plurisubharmonic function.
\end{proof}

\begin{remark}
\begin{enumerate}
\item
 For each irreducible bounded symmetric domain, $a$ is given as follows:
\begin{table}[h]
\begin{tabular}{c|c|c|c|c|c|c}
$\Omega$& $\Omega_{p,q}^I (p\leq q)$& $\Omega_n^{II}$& $\Omega_n^{III}$ &$\Omega_n^{IV}(n\geq 2)$& $\Omega_{16}^{V}$& $\Omega_{27}^{VI}$  \\\hline
$a$&$\frac{p+q}{2}$&$n-1$&$\frac{n+1}{2}$&$n\over 2$&6&9 
\end{tabular}
\end{table}
\item
The boundary component $B$ in the proof of Theorem \ref{main theorem2} is of the form 
$\{e^{i\theta}\}\times \Delta^\perp$ for some $\theta$.
\end{enumerate}
\end{remark}

\begin{corollary}
Let $E = M \times_\rho \Omega$ be a holomorphic fiber bundle over a compact K\"ahler manifold $M$
with an irreducible bounded symmetric domain fiber $\Omega$ where $\rho\colon \pi_1(M)\rightarrow \text{Aut}(\Omega)$ is a non-reductive representation.
Suppose that  $\rho$ is a maximal $1$-parabolic representation.
If there are two families of solutions of \eqref{heat equation}
which have subsequences that converge to 1-components $\{e^{i\theta_1}\}\times \Delta_1^\perp$ and $\{e^{i\theta_2}\}\times \Delta_2^\perp$ respectively for some $\theta_1,\,\theta_2$ with different $\Delta_1$ and $\Delta_2$,
then $E$ is weakly 1-complete.
\end{corollary}
\begin{proof}
Let $B_1$ and $B_2$ be the 1-components where the subsequences of the families of the solution of \eqref{heat equation} converges. Let $\psi_1$ and $\psi_2$ be the psh functions which are constructed in Theorem \ref{main theorem2} with respect to $B_1$ and $B_2$ respectively. Then $\psi_1 + \psi_2$ gives a psh exhaustion on $E$ since $\partial \Delta_1\times \Delta_1^\perp$ and $\partial \Delta_2 \times \Delta_2^\perp$ do not have intersection if $\Delta_1$ and $\Delta_2$ are different.
\end{proof}

From what follows we consider when $\Omega$ is the unit ball $\mathbb B^N$. In this case the boundary component $B$ is a point on the boundary and $$
c_{B} (\mathbb B^N) = \left\{E=(E_1, E_2)\in \frak m^- =\frak m^-_{\Pi-\Lambda,1} + \frak m_2^{\Lambda,-} \colon \text{Im } E_1 - \text{Re } F(E_2,E_2)>0\right\}.
$$
Moreover $M\times_{\rho_c} \frak m^- \cong M\times_{\rho_c} 
(\frak m^-_{\Pi-\Lambda,1} + \frak m_2^{\Lambda,-})$ is an affine bundle, i.e. transition maps are affine since $G_\Lambda = \{id\}$ which  yields $\rho_c(\pi_1(M))\subset L_2^\Lambda 
K_{\Pi-\Lambda, 1}^*  N^{\Lambda,-}$.
Let 
$$
\mu\colon \pi_1(M)\rightarrow  L_2^\Lambda  K^*_{\Pi-\Lambda,1} 
$$
denote the representation $\rho_c$ project to $L_2^\Lambda K^*_{\Pi-\Lambda,1} $ by ignoring translations.
Then $M\times_\mu \frak m^-$ is a holomorphic vector bundle.

The holomorphic vector bundle is said to be {\it polystable} when it is a direct sum of stable vector bundles and each component has the same slope. See \cite{Uhlenbeck_Yau_1986} for the definitions of stable vector bundles and their slope.
\begin{theorem}\label{theorem_ball}
Let $E = M \times_{\rho} \mathbb B^N$ be a holomorphic $\mathbb B^N$-fiber bundle over a compact K\"ahler manifold $M$
where $\rho\colon \pi_1(M)\rightarrow \text{Aut}(\Omega)$ is a non-reductive representation.
Suppose that  $M\times_{\mu} \frak m^-$ is a polystable vector bundle.
Then $M$ is weakly 1-complete.
\end{theorem}
\begin{proof}
Since  all Chern classes of the flat bundle vanish, by Uhlenbeck-Yau (\cite{Uhlenbeck_Yau_1986}) $M\times_{\mu} \frak m^-$
has a flat Hermitian structure, i.e. there exists a trivializing cover $\{U_\alpha\}$ of $M\times_{\rho_c} \frak m^- \to M$ with fiber coordinates $\{t_\alpha\}$ satisfying 
$$t_\alpha = a_{\alpha\beta} t_\beta + b_{\alpha\beta}
$$
where $a_{\alpha\beta}\in U(N)$ and $b_{\alpha\beta}\in \mathbb C^N$.
Moreover there exists a pluriharmonic functions $c_\alpha\colon U_\alpha\rightarrow \mathbb C^N$ such that $b_{\alpha\beta} = c_\alpha - a_{\alpha\beta}c_\beta$ on $U_\alpha\cap U_\beta$ since $M$ is K\"ahler (see \cite{Diederich_Ohsawa_1982}). Then the function $\phi:=|t_\alpha - c_\alpha|^2$ defines a psh exhaustion function on $M\times_{\rho_c} \frak m^- \to M$. 
Now define 
$$\psi_{c_B(\mathbb B^N)}([p,E]) := \phi([p,E]) 
- \log\text{dist}([p,E], \partial(M\times_{\rho_c}c_B(\Omega)))$$
where dist is induced from the flat Hermitian structure of $M\times_{\mu} \frak m^-$.
Then this is the desired psh exhaustion function.
\end{proof}

\begin{corollary}
Any $\mathbb B^2$ fiber bundle over a compact K\"ahler manifold is weakly 1-complete.
\end{corollary}
\begin{proof}
By Theorem \ref{main theorem1}, we only need to consider when $\rho$ is non-reductive. Since every boundary component of $\mathbb B^2$ is an 1-component and 
$M\times_{\mu} (\frak m^-_{\Pi-\Lambda, 1} + \frak m_2^{\Lambda,-})$ 
is a direct sum of flat line bundles (hence of degree zero)
$M\times_{\mu} \frak m^-_{\Pi-\Lambda, 1}$ and $M\times_{\mu}  \frak m_2^{\Lambda,-}$ by Lemma \ref{split} and Theorem \ref{Cayley transform} \eqref{adjoint}, the fiber bundle is weakly 1-complete by Theorem \ref{theorem_ball}.
\end{proof}

\begin{lemma}\label{split}
Let  $\mu\colon \pi_1(M)\to GL(V)$ be a representation of a vector space $V$. Suppose $V = V_1 + V_2$ and $\mu$ acts invariantly on $V_1$ and $V_2$. Then 
 $M\times_{\mu}V$ is isomorphic to $M\times_\mu V_1 \oplus M\times_\mu V_2$.
\end{lemma}
\begin{proof}
Since  $V_1$ and $V_2$ are invariant under the action of $\mu$, the transition functions of $M\times_{\mu}V$ are of the form 
$
\left( \begin{array}{cc}
g_{\alpha\beta}& 0\\
0& h_{\gamma\delta}
\end{array}\right)
$
which implies that $M\times_\mu V$ is isomorphic to $M\times _\mu V_1\oplus M\times_\mu V_2$.
\end{proof}

\section{$\Omega$-fiber bundles over compact quotients of BSDs}

\subsection{Hyperconvexity}
Let $\Gamma\subset \text{Aut}(\Omega)$ be a cocompact discrete subgroup of  $\text{Aut}(\Omega)$.
Then $\Gamma\setminus\Omega$ be a compact K\"ahler manifold with respect to 
the Bergman metric on $\Omega$.  Consider the diagonal action of $\Gamma$ 
on $\Omega\times\Omega$ defined by 
\begin{equation}\label{diagonal action}
\gamma(z,w) = (\gamma z, \gamma w).
\end{equation}
We denote the quotient manifold of $\Omega\times\Omega$ with respect to the action 
\eqref{diagonal action} by  $\Omega \times \Omega /\Gamma$. One can notice that 
$\Omega\times\Omega/\Gamma$ is an $\Omega$-fiber bundle over $\Gamma\setminus\Omega$.
For a generic norm $N_\Omega$ given in Section \ref{psh}, 
define \begin{equation}\nonumber
\delta(z,w) :=\frac{N_\Omega(z,z)N_\Omega(w,w)}{|N_\Omega(z,w)|^2}.
\end{equation}

\begin{theorem}\label{hyperconvex}
Let $\Omega$ be an irreducible bounded symmetric domain.
Then $\Omega\times \Omega/\Gamma$ is hyperconvex.
More precisely, $-\delta^{1/ r}$ with $\frac{1}{r}\leq \frac{1}{2\text{rank}(\Omega)}$ is a bounded psh exhaustion function.
\end{theorem}
\begin{proof}
Since $\delta$ is an invariant function on $\Omega\times\Omega$ under the 
diagonal action of $\text{Aut}(\Omega)$, $\delta$ is a well defined function 
on $\Omega\times\Omega /\Gamma$.
Since the proofs are similar, we will only show 
when $\Omega$ is type I or type IV domain.

{\bf Type I, $\Omega_{p,q}^I$:}
At $(0,w_0)$ with $w_0=\text{diag}(w_{11},\ldots, w_{pp})$ since we have 
\begin{equation}\nonumber
\begin{aligned}
\frac{\partial}{\partial w_{ij}} \log \delta(z,w) 
&= \frac{\partial}{\partial w_{ij}} \log
\frac{\det(I-z\overline z^t)\det(I-w\overline w^t)}
{\det(I-z\overline w^t)\det(I-w\overline z^t)}\\
&=\frac{1}{ \det(I- w\overline w^t)}\frac{\partial}{\partial w_{ij}} \det(I- w\overline w^t)
= \frac{-\overline w_{ii}\delta_{ij}}{1-|w_{ii}|^2}
\end{aligned}
\end{equation}
and 
\begin{equation}\nonumber
\begin{aligned}
\frac{\partial}{\partial z_{ij}} \log \delta(z,w) 
= \frac{\partial}{\partial z_{ij}} \log
\frac{\det(I-z\overline z^t)\det(I-w\overline w^t)}
{\det(I-z\overline w^t)\det(I-w\overline z^t)}
=-\frac{\partial}{\partial z_{ij}} \det(I- z\overline w^t)
= \overline w_{ij},
\end{aligned}
\end{equation}
one obtains
\begin{equation}\nonumber
\begin{aligned}
&\frac{\partial \log \delta}{\partial z_{ij}}
\frac{\partial \log \delta}{\partial \overline z_{kl}}
= \overline w_{ij} w_{kl}, \quad
\frac{\partial \log \delta}{\partial z_{ij}}
\frac{\partial \log \delta}{\partial \overline w_{kl}}
= \frac{-\overline w_{ii} w_{kk}\delta_{ij}\delta_{kl}}{1-|w_{kk}|^2},\quad\\
&\frac{\partial \log \delta}{\partial w_{ij}}
\frac{\partial \log \delta}{\partial \overline w_{kl}}
= \frac{\overline w_{ii} w_{kk}\delta_{ij}\delta_{kl}}{(1-|w_{ii}|^2)(1-|w_{kk}|^2)}.
\end{aligned}
\end{equation}
Therefore for $Z=\sum X_{ij}\frac{\partial}{\partial z_{ij}} + \sum Y_{ij} \frac{\partial}{\partial w_{ij}}$, by \eqref{ddlogpsi} and \eqref{plurisubharmonicity} we have 
\begin{equation}\nonumber
\begin{aligned}
r\frac{\partial\overline \partial (-\delta^{1/r})(Z,\overline Z)}{\delta^{1/r}} 
=& -\partial\overline\partial\log \delta(Z,\overline Z)
- \frac{1}{r}\partial \log \delta (Z)\wedge \overline\partial \log\delta(\overline Z)\\
=&\sum_{i,j=1}^p \left(|X_{ij}|^2 - 2\text{Re}(X_{ij}\overline Y_{ij}) + \frac{|Y_{ij}|^2}{(1-|w_{ii}|^2)(1-|w_{jj}|^2)} \right)\\
&-\frac{1}{r}
\left| \sum_{i=1}^p\overline w_{ii} X_{ii}-\sum_{i=1}^p\frac{\overline w_{ii} Y_{ii}}{1-|w_{ii}|^2}\right|^2  
+\sum_{i=1}^p\sum_{k=p+1}^q 
\left( \frac{|Y_{ik}|^2}{1-|w_{ii}|^2} +|X_{ij}|^2 -2\text{Re}(X_{ij}\overline Y_{ij})\right).
\end{aligned}
\end{equation}
Since
\begin{equation}\nonumber
\begin{aligned}
&\sum_{i=1}^p \left(|X_{ii}|^2 - 2\text{Re}(X_{ii}\overline Y_{ii}) + \frac{|Y_{ii}|^2}{(1-|w_{ii}|^2)^2} \right)
-\frac{1}{r}
\left| \sum_{i=1}^p\left(\overline w_{ii} X_{ii}-\frac{\overline w_{ii} Y_{ii}}{1-|w_{ii}|^2}\right)\right|^2  \\
&\geq \sum_{i=1}^p \left(|X_{ii}|^2 - 2\text{Re}(X_{ii}\overline Y_{ii}) + \frac{|Y_{ii}|^2}{(1-|w_{ii}|^2)^2} \right)
-\frac{2p}{r}\sum_{i=1}^p\left(\left|\overline w_{ii} X_{ii}\right|^2 + \left|\frac{\overline w_{ii} Y_{ii}}{1-|w_{ii}|^2}\right|^2\right)\\
&\geq \sum_{i=1}^p \left| \sqrt{1-|w_{ii}|^2}X_{ii} - \frac{Y_{ii}}{\sqrt{1-|w_{ii}|^2}}\right|^2 \geq 0,
\end{aligned}
\end{equation}
we obtain the proposition for $\Omega_{p,q}^I$.

{\bf Type IV:}
At $(0,w_0)$ with $w_0=(w_1, w_2,0,\ldots, 0)$
where $w_1 = \lambda(\zeta_1 + \zeta_2)$, 
$w_2 = i\lambda ( \zeta_1 - \zeta_2)$ as in the proof of Lemma \ref{lemma_psh} for type IV domain, since we have 
\begin{equation}\nonumber
\begin{aligned}
\frac{\partial}{\partial w_{j}} \log \delta
&=\frac{1}{1-2w\overline w^t + \left| ww^t\right|^2}\frac{\partial}{\partial w_{j}} (1-2w\overline w^t + \left| ww^t\right|^2)
= \frac{-2\overline w_j + 2w_j \overline{ww^t}}{1-2w\overline w^t + \left| ww^t\right|^2}
\end{aligned}
\end{equation}
and 
\begin{equation}\nonumber
\begin{aligned}
\frac{\partial}{\partial z_{j}} \log \delta
=-\frac{\partial}{\partial z_{j}} (1-2z\overline w^t + zz^t \overline{ww^t})
= 2\overline w_{j},
\end{aligned}
\end{equation}
one obtains for $Z=\sum X_{j}\frac{\partial}{\partial z_{j}} + \sum Y_{j} \frac{\partial}{\partial w_{j}}$, 
\begin{equation}\label{type4_hyperconvex}
\begin{aligned}
r\frac{\partial\overline \partial (-\delta^{1/r})(Z,\overline Z)}{\delta^{1/r}} 
=& -\partial\overline\partial\log \delta(Z,\overline Z)
- \frac{1}{r}\partial \log \delta (Z)\wedge \overline\partial \log\delta(\overline Z)\\
\
=&\sum_{j=1}^n 2|X_j|^2 - 4\text{Re}\sum_{j=1}^n X_j\overline Y_j 
+M_{11}|Y_1|^2 + M_{12}Y_1\overline Y_2 + M_{21} Y_2\overline Y_1 + M_{22}|Y_2|^2\\
&-\frac{1}{r}\left| \sum_{j=1}^2 \left( 2 \overline w_j X_j + \frac{-2\overline w_j + 2w_j \overline{ww^t}}{1-2w\overline w^t + \left| ww^t\right|^2}Y_j\right)\right|^2.
\end{aligned}
\end{equation}
By substituting $\zeta_j$ variable, we have 
\begin{equation}\nonumber
\begin{aligned}
&\sum_{j=1}^2 \left(  \overline w_j X_j + \frac{-\overline w_j + w_j \overline{ww^t}}{1-2w\overline w^t + \left| ww^t\right|^2}\right)\\
&= \overline\lambda\left(
\overline \zeta_1(X_1-iX_2) + \overline \zeta_2(X_1+iX_2) 
-\frac{\overline \zeta_1}{S^{IV}_n}(1-|\zeta_2|^2)(Y_1 - iY_2) 
-\frac{\overline \zeta_2}{S^{IV}_n}(1-|\zeta_1|^2) ( Y_1+iY_2)
\right)
\end{aligned}
\end{equation}
and 
\begin{equation}\nonumber
\begin{aligned}
&M_{11}|Y_1|^2 + M_{12}Y_1\overline Y_2 + M_{21} Y_2\overline Y_1 + M_{22}|Y_2|^2\\
&=\frac{1}{S^{IV}_n}\left(
(1-|\zeta_1|^2)^2 |Y_1 + iY_2|^2 + (1-|\zeta_2|^2)^2 |Y_1-iY_2|^2
\right)
\end{aligned}
\end{equation}
by the equalities 
\begin{equation}\nonumber
\begin{aligned}
&M_{11}=M_{22}=\frac{1}{S^{IV}_n} 
\left( (1-|\zeta_1|^2)^2 + (1-|\zeta_2|^2)^2\right),\\
&M_{12}=-M_{21}=\frac{i}{S^{IV}_n} 
\left( (1-|\zeta_2|^2)^2 - (1-|\zeta_1|^2)^2\right),\\
& M_{12}Y_1\overline Y_2 + M_{21} Y_2\overline Y_1
= \frac{iM_{12}}{2}\left(
|Y_1+iY_2|^2 - |Y_1-iY_2|^2\right).
\end{aligned}
\end{equation}
Therefore the equation \eqref{type4_hyperconvex} is greater than or equal to 
\begin{equation}\nonumber
\begin{aligned}
&(1-|\zeta_1|^2)|X_1 - iX_2|^2 + (1-|\zeta_2|^2)|X_1 + iX_2|^2  - 4\text{Re}X_1\overline Y_1 - 4\text{Re} X_2\overline Y_2 + \frac{|Y_1 - iY_2|^2}{1-|\zeta_1|^2} +\frac{|Y_1 + iY_2|^2}{1-|\zeta_2|^2} \\
&=\left| \sqrt{1-|\zeta_1|^2}(X_1-iX_2) - \frac{Y_1-iY_2}{\sqrt{1-|\zeta_1|^2}}\right|^2
+
\left| \sqrt{1-|\zeta_2|^2}(X_1+iX_2) - \frac{Y_1+iY_2}{\sqrt{1-|\zeta_2|^2}}\right|^2
\end{aligned}
\end{equation}
and hence $-\delta^{1/r}$ is psh.
\end{proof}

Recall the definition of the Diederich-Fornaess index:
for a domain $D$ in a complex manifold of dimension $n$, $n\geq 2$ with smooth boundary,
the Diederich-Fornaess index of $D$ is defined by 
$$ \sup\left\{ \mu\in (0,1) : -\partial\overline\partial (-\nu)^\mu >0 \text{ on } D\right\}$$
where the supremum is taken over all defining function $\nu$ of $D$.
\begin{corollary}\label{1/2}
The Diederich-Fornaess index of $\mathbb B^n\times\mathbb B^n/\Gamma$ in 
$ \mathbb{B}^n\times\mathbb{CP}^n/\Gamma$ is $1/2$.
\end{corollary}
\begin{proof}
Since $-\delta = -\frac{(1-|z|^2)(1-|w|^2)}{|1-z\overline w|^2}$ is invariant with respect to 
the action of $\Gamma$, it gives a real analytic defining function of 
$\mathbb B^n\times\mathbb B^n/\Gamma \subset 
\mathbb B^n\times\mathbb{CP}^n/\Gamma$.
Adachi-Brinkschulte(\cite[Main Theorem]{Adachi_Brinkschulte_2015}) 
and Fu-Shaw(\cite{Fu_Shaw_2016}) proved independently that 
for a relatively compact domain with $C^3$ boundary in a complex manifold of dimension $N$, 
if the Levi form of the domain has at least 
$k$ zero eigenvalues everywhere on the boundary with $0\leq k \leq N-1$, then the 
Diederich-Fornaess index should be less or equal to $\frac{N-k}{N}$.
Since the Levi form of $\mathbb B^n\times \mathbb B^n/\Gamma$ 
has at least $n$ number of zero eigenvalues, the Diederich-Fornaess index
should be less than or equal to 1/2.
On the other hand by Theorem \ref{hyperconvex} the index should be greater or equal to 1/2 and 
hence we complete the proof.

\end{proof}

\subsection{$k$-twisted BSDs}
Let us consider the diagonal action of $\rho\colon \Gamma \rightarrow \text{Aut}(\Omega)$
on $\Omega^{k} :=\Omega\times\dots\times\Omega$ given by
$\gamma (z_1,\ldots, z_k) = (\gamma z_1, \dots, \gamma z_k)$.
Let $\Omega^k/\Gamma$ be the quotient of $\Omega^k$ by this diagonal action.
Then it is a holomorphic $\Omega^{k-1}$-fiber bundle over $\Omega/\Gamma$.
Define a function $\psi_k$ on $\Omega^k$ by 
\begin{equation}\nonumber
\psi_k (z) := \left | \frac{\prod^k_{j=1} K_\Omega(z_j, z_j)}{K_\Omega(z_1,z_{2})K_\Omega(z_2,z_3)\cdots K_\Omega(z_{k-1},z_k)K_\Omega(z_k,z_1)}\right |^2.
\end{equation}
Then $\psi_2 = \psi_\Omega^2$, where $\psi_\Omega$ is given in \eqref{psi},
and one has 
\begin{equation}\nonumber
\psi_k(z_1,\ldots z_k )= \psi_\Omega(z_1,z_2)\,\cdots\,\psi_\Omega(z_{k-1},z_k)\,\psi_\Omega(z_k,z_1).\end{equation}
Since $$\partial \overline\partial \log\psi_k(z_1,\ldots,z_k) 
= \sum_{j=1}^{k-1} \partial\overline\partial \log \psi_\Omega(z_j,z_{j+1})  + \partial\overline\partial \log \psi_\Omega(z_k, z_1)\geq 0,$$
we have the following: 
\begin{corollary}
For any irreducible bounded symmetric domain  $\Omega$ and $k\geq 2$, $\Omega^k/\Gamma$ is hyperconvex.
\end{corollary}

\subsection{Steinness}
Let $\Gamma\subset G$ be a cocompact discrete subgroup of  $\text{Aut}(\Omega)$.
Then $\Gamma\setminus\Omega$ is a compact K\"ahler manifold with respect to the metric induced from
the Bergman metric on $\Omega$.  Now consider the diagonal action of $\Gamma$ 
on $\Omega\times\Omega$ defined by
\begin{equation}\label{twisted action}
\gamma(z,w) = \left(\gamma z, \overline{\gamma \overline w}\right).
\end{equation}
Denote by $\Omega\times\Omega/ \overline \Gamma$ 
the quotient manifold of $\Omega\times\Omega$ by the action \eqref{twisted action}.
Then $\Omega \times \Omega /\overline\Gamma$ is an $\Omega$-fiber bundle 
over $\Gamma\setminus\Omega$.
Now consider the function on $\Omega\times\Omega$ defined by
\begin{equation}\nonumber
\overline \psi_\Omega(z,w):= \psi_\Omega(z,\overline w)
=\frac{K_\Omega(z,z)K_\Omega (w,w)}{|K_\Omega(z,\overline w)|^2}.
\end{equation}
Since 
\begin{equation}\nonumber
\begin{aligned}
\overline\psi_\Omega \left(\gamma(z,w)\right) 
= \overline\psi_\Omega (\gamma z, \overline{\gamma \overline w})
&= \frac{K_\Omega\left(\gamma z,\gamma z\right)K_\Omega 
\left( \overline{\gamma \overline w}, \overline{\gamma \overline w}\right)}
{\left|K_\Omega\left(\gamma z,\overline{\overline{\gamma \overline w}}\right)\right|^2}\\
&= \frac{K_\Omega\left(\gamma z,\gamma z\right)K_\Omega 
\left( {\gamma \overline w}, {\gamma \overline w}\right)}
{\left|K_\Omega\left(\gamma z,\gamma \overline w\right)\right|^2}
=\overline\psi_\Omega (z,w),
\end{aligned}
\end{equation}
$\overline \psi_\Omega$ induces a function on $\Omega\times\Omega/\overline\Gamma$.
Hence $\overline \delta(z,w) := \frac{N_\Omega(z,z)N_\Omega (w,w)}
{\left|N_\Omega(z,\overline w)\right|^2}$ also induces a function on $\Omega\times\Omega/\overline\Gamma$ 
and it is an exhaustion function.
\begin{theorem}\label{steinness}
Let $\Omega$ be an irreducible bounded symmetric domain.
Then $\Omega\times \Omega/\overline\Gamma$ admits a bounded strictly psh exhaustion function.
More precisely, $-\overline\delta^{1/ r}$ with $\frac{1}{r}\leq \frac{1}{2\text{rank}(\Omega)}$ is a bounded strictly psh exhaustion function.
\end{theorem}
\begin{proof}
{\bf Type I, $\Omega_{p,q}^I$:}
Since $N_\Omega(z,\overline w)$ is holomorphic in the $z_i$'s and $w_i$'s, we have 
\begin{equation}\nonumber
-\partial\overline\partial \log\overline \delta
= \left( \begin{array}{cc}
I_{pq}& 0 \\
0& -\partial_w\bar\partial_w\log \det(I-w\overline w^t)
\end{array}\right)
\end{equation}
at $(0,w_0)$ with $w_0=\text{diag}(w_{11},\ldots, w_{pp})$, and hence for a nonzero vector $Z= \sum X_{jk}\frac{\partial}{\partial z_{jk}} + \sum Y_{jk}\frac{\partial}{\partial w_{jk}}$ at $(0,w_0)$ we obtain
 \begin{equation}\nonumber
\begin{aligned}
r\frac{\partial\overline \partial (-\bar\delta^{1/r})}{\bar\delta^{1/r}} 
&= -\partial\overline\partial\log \bar \delta
- \frac{1}{r}\partial \log \bar\delta \wedge \overline\partial \log\overline\delta\\
&=\sum_{j,k=1}^p \left(|X_{jk}|^2 + \frac{|Y_{jk}|^2}{(1-|w_{kk}|^2)(1-|w_{jj}|^2)} \right)\\
&\quad-\frac{1}{r}
\left| \sum_{j=1}^p w_{jj} X_{jj}-\sum_{j=1}^p\frac{\overline w_{jj} Y_{jj}}{1-|w_{jj}|^2}\right|^2  
+\sum_{j=1}^p\sum_{k=p+1}^q 
\left( \frac{|Y_{jk}|^2}{1-|w_{jj}|^2} +|X_{jk}|^2 \right)\\
&\geq \sum_{j, k=1, j\neq k}^p \left(|X_{jk}|^2 + \frac{|Y_{jk}|^2}{(1-|w_{kk}|^2)(1-|w_{jj}|^2)} \right) \\
&\quad +\sum_{j=1}^p\left( (1-|w_{jj}|^2)|X_{jj}|^2 + \frac{|Y_{jj}|^2}{1-|w_{jj}|^2}\right)
+ \sum_{j=1}^p\sum_{k=p+1}^q 
\left( \frac{|Y_{jk}|^2}{1-|w_{jj}|^2} +|X_{jk}|^2 \right) 
>0.
\end{aligned}
\end{equation}
Therefore $-\overline \delta^{1/r}$ is strictly psh exhaustion function.

{\bf Other cases:} We omit the proof since we can apply a same way for type I domains.
\end{proof}

By a proof similar to that of Corollary \ref{1/2}, we obtain the following:
\begin{corollary}
The Diederich-Fornaess index of $\mathbb B^n\times\mathbb B^n/\overline\Gamma$ in 
$ \mathbb{B}^n\times\mathbb{CP}^n/\overline \Gamma$ is $1/2$.
\end{corollary}

\begin{remark}
In case $\Omega$ is the unit disc in $\mathbb C$, Theorem \ref{steinness} was 
proved by Adachi in \cite{Adachi_2017}.
\end{remark}

\end{document}